\documentclass[a4paper]{amsart}
\usepackage[utf8]{inputenc}

\title{Metric Scott analysis}

\author{Ita{\"\i} Ben Yaacov}
\address{Ita{\"\i} Ben Yaacov \\
  Univ Lyon \\
  Université Claude Bernard Lyon 1 \\
  Institut Camille Jordan, CNRS UMR 5208 \\
  43 boulevard du 11 novembre 1918 \\
  69622 Villeurbanne \textsc{cedex} \\
  France}
\urladdr{\url{http://math.univ-lyon1.fr/~begnac/}}

\author{Michal Doucha}
\address{
  Michal Doucha \\
  Laboratoire de Mathématiques de Besançon \\
  Université de Franche-Comté \\
  16, route de Gray \\
  25030 Besançon \textsc{cedex} \\
  France}
\email{doucha@math.cas.cz}

\curraddr{Institute of Mathematics CAS \\
\v{Z}itn\'{a} 25 \\
115 67 Praha 1 \\
Czech Republic}

\author{Andr\'e Nies}
\address{
  Andr\'e Nies \\
  Department of Computer Science \\
  University of Auckland \\
  Private Bag 92019 \\
  Auckland \\
  New Zealand}
\email{andre@cs.auckland.ac.nz}

\author{Todor Tsankov}
\address{
  Todor Tsankov \\
  Institut de Math{\'e}matiques de Jussieu--PRG \\
  Universit\'e Paris 7, case 7012 \\
  75205 Paris \textsc{cedex} 13 \\
  France}
\email{todor@math.univ-paris-diderot.fr}

 \date{\today}

\subjclass[2010]{Primary 03C75, 03E15}
\keywords{continuous logic, infinitary logic, Scott rank, Scott sentence, López-Escobar theorem, Borel equivalence relations, Gromov--Hausdorff distance, Kadets distance, CLI groups}

\usepackage[T1]{fontenc}

\usepackage[osf]{mathpazo}

\usepackage{eucal}

\usepackage{hyperref}
\usepackage{my-macros}

\usepackage{enumitem}
\setlist[enumerate,1]{label=(\roman*), font=\normalfont}

\usepackage[initials, shortalphabetic]{amsrefs}

\newcommand{\Lomo}{\cL_{\omega_1\omega}}
\newcommand{\Lfin}{\cL_{\omega\omega}}

\newcommand{\dGH}{d_{\mathrm{GH}}}
\newcommand{\deGH}{d_{\mathrm{eGH}}}
\newcommand{\dK}{d_\mathrm{K}}
\newcommand{\dH}{d_\mathrm{H}}
\newcommand{\deK}{d_\mathrm{eK}}

\newcommand{\OmegaL}{\Omega_{\mathrm{L}}}
\newcommand{\OmegaU}{\Omega_{\mathrm{U}}}
\newcommand{\Lip}{\mathrm{Lip}}

\newcommand{\dotminus}{\mathbin{\ooalign{\hss\raise1ex\hbox{.}\hss\cr
  \mathsurround=0pt$-$}}}
\newcommand{\half}{\tfrac12}

\DeclareMathOperator{\qr}{qr} % quantifier rank
\DeclareMathOperator{\BC}{BC} % Baire class

\DeclareMathOperator*{\qsup}{sup\,}
\DeclareMathOperator*{\qinf}{inf\,}

\begin{document}

\begin{abstract}
We develop an analogue of the classical Scott analysis for metric structures and infinitary continuous logic. Among our results are the existence of Scott sentences for metric structures and a version of the López-Escobar theorem. We also derive some descriptive set theoretic consequences: most notably, that isomorphism on a class of separable structures is a Borel equivalence relation iff their Scott rank is uniformly bounded below $\omega_1$. Finally, we apply our methods to study the Gromov--Hausdorff distance between metric spaces and the Kadets distance between Banach spaces, showing that the set of spaces with distance $0$ to a fixed space is a Borel set.
\end{abstract}

\maketitle

\setcounter{tocdepth}{1}
\tableofcontents

%%%%%%%%%%%%%%%%%%%%%%%%%%%%%%%%%%%%%%%%%%%%%%%%%%

\section{Introduction}
\label{sec:introduction}

Two of the foundational results in classical infinitary logic are Scott's theorem~\cite{Scott1965}, producing  for every countable structure a canonical sentence that uniquely describes it up to isomorphism, and the theorem of López-Escobar~\cite{Lopez-Escobar1965}, characterizing the isomorphism-invariant Borel sets of models. These results and the techniques developed around them, \emph{Scott analysis} (based on the back-and-forth method of Ehrenfeucht and Fraïssé) and the \emph{Vaught transforms} (introduced by Vaught~\cite{Vaught1974} to give a new proof of the López-Escobar theorem) have become a cornerstone of infinitary model theory as well as of the descriptive set theoretic study of the complexity of isomorphism of countable models. See, e.g., Gao~\cite{Gao2009a} for an exposition of the general theory, and Hjorth--Kechris~\cite{Hjorth1996} and Hjorth--Kechris--Louveau~\cite{Hjorth1998} for some more detailed results. The notion of \emph{Scott rank}, an ordinal that measures the model-theoretic complexity of structures, is also an indispensable tool.

The goal of this paper is to develop a parallel theory for infinitary continuous logic. In recent years, there has been a lot of activity in first order continuous logic (see \cite{BenYaacov2008} for a gentle introduction) and it turns out that much of the classical first order model theory extends to this setting, often with interesting twists. Perhaps more importantly, it seems that continuous logic is the ``correct'' setting for applying model-theoretic ideas to functional analysis and operator algebras, areas that have been hitherto far removed  from model theory (see, for example, \cite{Farah2014} and the references therein).

Some progress has been made towards the study of metric structures using classical infinitary logic (see, for example, \cite{Doucha2014}).
However, it seems that classical logic is too expressive in this setting and continuous logic is more appropriate for descriptive set theoretic applications. We discuss the connections of our approach with classical logic in Section~\ref{sec:classical-logic}.

An extended  form of continuous logic, called infinitary here,  allows  connectives to act on certain infinite collections of formulas. It  was introduced in~\cite{BenYaacov2009a}, where the authors obtain some applications to Banach space theory. Two other papers that focus on various versions of the omitting types theorem for   infinitary continuous  logic (that we also use here) are Caicedo and Iovino~\cite{Caicedo2014} and Eagle~\cite{Eagle2014}. However, their framework is somewhat more general (for example, they do not always require that structures be complete), while we keep the setting from \cite{BenYaacov2009a} as it seems to be the most relevant to  our purposes.

% In infinitary logic, one loses some of the pleasant features of classical model theory (most notably, compactness) but the increased expressive power of the language 

Our results are inspired by their classical counterparts but, as is usual in continuous logic, new difficulties and interesting phenomena appear that have no analogue in the classical setting. On a philosophical level, this is perhaps best explained by descriptive set theory: the isomorphism equivalence relation of classical countable models (which, in view of Scott's results, is one of the main objects of study) is strictly less complicated than isomorphism of separable metric structures. More precisely, the former is a universal orbit equivalence relation of an action of $S_\infty$ and the latter is (bi-reducible with) the universal orbit equivalence relation for Polish group actions \cite{Elliott2013}; by Hjorth's results \cite{Hjorth2000} on turbulence, the latter is strictly more complicated.

The basis for   classical Scott analysis is given by the back-and-forth equivalence relations $\equiv_\alpha$ (originally defined by Fraïssé) indexed by the countable ordinals~$\alpha$. These can be considered as Borel approximations of the analytic equivalence relation of isomorphism. The first novelty in the metric setting is that these equivalence relations are replaced by pseudo-distances $r_\alpha$ (i.e., distinct points can have distance $0$) that measure how different two tuples,  of the same length  and possibly coming from different structures, are. These pseudo-distances naturally give rise to equivalence relations $E_\alpha$: $\bar a \eqrel{E_\alpha} \bar b \iff r_\alpha(\bar a, \bar b) = 0$, and $E_\infty = \bigcap_\alpha E_\alpha$. The inductive definition is mostly uneventful, apart from the base case. For classical structures, two tuples are declared to be $\equiv_0$-equivalent if all quantifier-free formulas agree on them; in the metric situation, we would like to define $r_0$ as the supremum of the difference of the values that quantifier-free formulas take when evaluated on the tuples. This approach meets an immediate obstacle: the difference can be arbitrarily amplified by scaling formulas by a multiplicative constant. Thus we are led to consider formulas with a fixed modulus of continuity and in order to organize this, we introduce the notion of a \df{weak modulus of continuity} (denoted by $\Omega$) that controls what formulas we are allowed to use in the definition of $r_0$. It turns out that the weak modulus is an additional parameter in the construction that has no analogue in the classical setting;  by varying it one obtains different equivalence relations $E_\infty$ at the end.

The basic results of the general theory are independent of the choice of a weak modulus. Our first theorem is that the pseudo-distances $r_\alpha$ capture exactly the expressive power of formulas of quantifier rank at most $\alpha$.
\begin{theorem}
  \label{th:intro:rank}
  Let $A, B$ be metric structures in the same signature and let $\bar a \in A^n, \bar b \in B^n$. Then, for any $\alpha < \omega_1$,
  \begin{equation*}
    r_\alpha(A \bar a, B \bar b) = \sup_\phi |\phi^A(\bar a) - \phi^B(\bar b)|,
  \end{equation*}
  where the sup is taken over all $n$-ary $\Omega$-formulas of quantifier rank $\alpha$.
\end{theorem}

We also obtain analogues of Scott's results from the classical setting.
\begin{theorem}
  \label{th:intro-Scott}
  For every separable structure $A$, there exists a $\Lomo$-sentence $\sigma_A$ (the \df{Scott sentence} of $A$) such that for every separable structure $B$,
  \begin{equation*}
    A \eqrel{E_\infty} B \iff \sigma_A(B) = 0.
  \end{equation*}
  Moreover, the quantifier rank of $\sigma_A$ is equal to the Scott rank of $A$ plus $\omega$. In particular, the $E_\infty$-class of $A$ is Borel and its complexity is bounded by the Scott rank of $A$.
\end{theorem}

With a certain choice of the weak modulus $\Omega$ (which we call \df{universal}), one obtains isomorphism as $E_\infty$-equivalence, exactly as in the classical case. All separable metric structures for a fixed language can be seen as points in a Polish space,  as explained   in detail in Section~\ref{sec:space-polish-struct}. Applying Theorem~\ref{th:intro-Scott} tells us that isomorphism classes of separable metric structures are Borel (which is one of the classical applications of Scott sentences). This result, however, is not new: it can be deduced from \cite{Elliott2013}, where the authors prove that isomorphism is reducible to the orbit equivalence relation of a group action. Our approach, however, gives more detailed information and bounds for the Borel complexity of the equivalence class in terms of the Scott rank.

Our next theorem characterizes exactly when the isomorphism equivalence relation is Borel (again inspired by a similar result in the classical setting).
\begin{theorem}
  \label{th:intro:Borel-bounded}
  Let $\cong$ denote the isomorphism equivalence relation of separable structures (in a fixed signature) and let $X$ be an $\cong$-invariant Borel set of structures. Then the following are equivalent:
  \begin{enumerate}
  \item $\cong|_X$ is Borel;
  \item the supremum of the Scott ranks of the elements of $X$ is bounded below $\omega_1$.
  \end{enumerate}
\end{theorem}
In particular, this theorem provides a new method to show that certain isomorphism equivalence relations are not Borel, so long as one is able to calculate the Scott ranks.

An important connection between infinitary logic and descriptive set theory is provided by the López-Escobar theorem which asserts that the $\sigma$-algebra of Borel, $\cong$-invariant sets in the space of models coincides with the algebra of sets definable by $\Lomo$-sentences. We obtain an analogue of this theorem in the continuous setting.
\begin{theorem}
  \label{th:intro:LE}
  Let $U$ be a bounded, Borel function on the space of separable models that is invariant under isomorphism. Then there exists a $\Lomo$-sentence $\phi$ such that $U(A) = \phi^A$ for every structure $A$.
\end{theorem}
Our proof of this theorem is based on \df{Vaught transforms} \cite{Vaught1974}. However, as we do not have a group action readily available, we develop the transforms in a  different setting, better adapted to our situation.

A theorem similar to Theorem~\ref{th:intro:LE} was independently and simultaneously proved by Coskey and Lupini~\cite{Coskey2016}. The main difference between their approach and ours is that they only consider structures with  universe   the Urysohn sphere $\bU_1$. In that case, the equivalence relation of isomorphism is given by the action of $\Iso(\bU_1)$ and Vaught transforms can be used directly.

With a different natural choice of a weak modulus (the \df{$1$-Lipschitz} one), the  pseudo-distance $r_\infty = \sup_\alpha r_\alpha$ defines a coarser equivalence relation than isomorphism and specializes to Gromov--Hausdorff distance for metric spaces and Kadets distance for Banach spaces. This approach to the Gromov--Hausdorff distance has the advantage that it does not require embeddings into a third structure and is defined for arbitrary metric structures, even ones that do not have amalgamation (and where the original definition is not applicable). Combining this with our general results, we obtain the following.
\begin{cor}
  \label{c:intro:GH-Borel}
  For a  Polish metric space $A$,  the set of Polish metric spaces such that the  Gromov--Hausdorff distance to $A$ is $0$ is Borel. A similar fact holds  for Banach spaces and the Kadets distance.
\end{cor}

Finally, generalizing a theorem of Gao~\cite{Gao1998} from the classical setting, we characterize the separable structures with a  Scott sentence that  has only separable models (that is, it is \df{absolutely categorical}).
\begin{theorem}
  \label{th:intro:Gao}
  Let $A$ be a separable metric structure. Then the following are equivalent:
  \begin{enumerate}
  \item The Scott sentence of $A$ only has separable models;
  \item The left uniformity of the group $\Aut(A)$ is complete.
  \end{enumerate}
\end{theorem}

\medskip \noindent \textbf{Acknowledgements.} We would like to thank Christian Rosendal for some useful discussions and Maciej Malicki for providing a reference. Part of the work on this project was carried out at Research Centre Coromandel in New Zealand. Continuing work   was carried out at the Erwin Schrödinger Institute in Vienna during the thematic program on \emph{Measured group theory};  we are grateful to the Institute and the organizers for the excellent working conditions they provided. This research was partially supported by the Marsden fund of New Zealand 13-UOA-287, the ANR project GrupoLoco (ANR-11-JS01-008), the ANR project GAMME (ANR-14-CE25-0004), the ERC grant ANALYTIC (no. 259527), and the Franche--Comt\'e region.

We are also grateful to the anonymous referee for a careful reading of the paper and some useful suggestions.

%%%%%%%%%%%%%%%%%%%%%%%%%%%%%%%%%%%%%%%%%%%%%%%%%% 

\section{Continuous infinitary logic}
\label{sec:cont-infin-logic}

\subsection{Moduli of continuity}
\label{ssec:moduli-continuity}

An important feature of classical $\Lomo$-logic is that when one forms an infinite conjunction (or disjunction) of the formulas $\set{\phi_i(\bar x) : i \in \N}$, all $\phi_i$ share the same finite set of free variables $\bar x$; as a result, every $\Lomo$-formula has only finitely many free variables. The analogous uniformity condition in infinitary continuous logic is ensured by mandating that all $\phi_i$ obey the same continuity modulus. This ensures that the interpretations of all formulas are uniformly continuous functions (with a modulus that can be determined syntactically). To formalize this, we start with several basic definitions and facts about moduli of continuity.

\begin{defn}
  \label{dfn:Modulus}
  Let $n$ be a natural number or $\N$. A \emph{modulus of arity $n$} is a function $\Delta \colon [0,\infty)^n \rightarrow [0,\infty)$ that is:
  \begin{enumerate}
  \item \label{item:Modulus1}
    non-decreasing, subadditive, vanishing at zero:
    \begin{equation*}
      \Delta(\delta) \leq \Delta(\delta + \delta') \leq \Delta(\delta) + \Delta(\delta'), \qquad \Delta(0) = 0;
    \end{equation*}
    
  \item continuous.
  \end{enumerate}

  A \df{weak modulus} is a function $\Omega \colon [0, \infty)^\N \to [0, \infty]$ that satisfies \ref{item:Modulus1} and is

  \begin{itemize}
  \item[(ii')] lower semi-continuous in the product topology and separately continuous in each argument.
  \end{itemize}
\end{defn}

The main use of a modulus is to measure the uniform continuity of a function defined on a product of finitely many metric spaces.
\begin{defn}
  Let $\Delta$ be an $n$-ary modulus, and let $X = \prod_{i < n} X_i$ be a product of metric spaces.
  On $X^2$, define
  \begin{equation*}
    d^\Delta(x,y) = \Delta \bigl(d^{X_i}(x_i,y_i) : i \in n \bigr).
  \end{equation*}
  If $Z$ is another metric space and $f\colon X \rightarrow Z$ is a map, we say that $f$ \df{respects} (or \df{obeys}) $\Delta$ if for all $x, y \in X$, we have
  \begin{equation*}
    d_Z( f(x), f(y) ) \leq d^\Delta(x, y).
  \end{equation*}
\end{defn}

The conditions in the definition of a modulus are chosen in such a way that $d^\Delta$ is a continuous pseudo-distance on any product of metric spaces. If $\Delta$ is moreover \df{faithful}, i.e., $\Delta(\delta) = 0$  implies that $\delta = 0$, then $d^\Delta$ is a distance compatible with the product uniform structure.

If $K \sub \R^n$ is a product of compact intervals and $f \colon K \to \R$ is a continuous function, we define its modulus of continuity $\Delta_f$ by
\begin{equation}
  \label{eq:ModulusFunction}
  \Delta_f(\delta) = \sup \set{|f(x) - f(y)| : x, y \in K, |x_i - y_i| \leq \delta_i}.
\end{equation}
$\Delta_f$ is the least modulus of continuity that $f$ obeys.

The main purpose of weak moduli is to control the uniform continuity of formulas and organize together an infinite collection of moduli of different arities. A weak modulus will never be used directly but rather via its traces on finite products. If $\Omega \colon [0, \infty)^\N \to [0, \infty]$ is a weak modulus and $n \in \N$, define the \emph{truncation} $\Omega|_n \colon [0, \infty)^n \to [0, \infty]$ by
\begin{equation*}
  \Omega|_n(\delta_0, \ldots, \delta_{n-1}) = \Omega(\delta_0, \ldots, \delta_{n-1}, 0, 0,  \ldots).
\end{equation*}

A weak modulus $\Omega$ is \emph{shift-increasing} if for every sequence $i_0 < i_1 < \cdots$ of natural numbers and every $\delta \in [0, \infty)^\N$, we have $\Omega(\delta) \leq \Omega(\delta')$, where $\delta'_{i_p} = \delta_p$ and $\delta'_k = 0$ if $k \notin \set{i_0, i_1, \ldots}$. All natural weak moduli that we have in mind satisfy this condition; however, it is only used in one place in the general theory (Proposition~\ref{p:BackAndForth}) and we have preferred to keep it as a separate hypothesis where necessary rather than make it part of the definition of a weak modulus.

The following lemma clarifies the connection between weak moduli and moduli.
\begin{lemma}
  \label{l:moduli-properties}
  Let $\Omega \colon [0, \infty)^\N \to [0, \infty]$ be a weak modulus. Then all truncations of $\Omega$ are moduli and $\Omega$ is determined by its truncations:
  \begin{equation}
    \label{eq:l:moduli-properties}
      \Omega(\delta_0,\delta_1, \ldots) = \sup_n \Omega|_n(\delta_0,\ldots,\delta_{n-1}).
    \end{equation} 
\end{lemma}
\begin{proof}
  Since $\Omega|_n$ is lower semi-continuous, it suffices to show that it is upper semi-continuous. Fix $\delta \in [0, \infty)^n$ and $s \in \R$ such that $\Omega|_n(\delta) < s$. Using that $\Omega|_n$ is separately continuous in each variable, find consecutively $\gamma_0, \ldots, \gamma_{n-1}$ such that
  \begin{equation*}
    \Omega|_n(\delta) \leq \Omega|_n(\delta_0 + \gamma_0, \delta_1, \ldots, \delta_{n-1}) \leq \cdots \leq \Omega|_n(\delta_0 + \gamma_0, \ldots, \delta_{n-1} + \gamma_{n-1}) < s.
  \end{equation*}
  As $\Omega|_n$ is monotone, this shows that $\set{\delta \in [0, \infty)^n : \Omega|_n(\delta) < s}$ is open, completing the proof.

  For \eqref{eq:l:moduli-properties}, the inequality $\geq$ follows from the monotonicity and $\leq$ follows from lower semi-continuity.
\end{proof}

Thus, given a metric space $(X, d)$, each weak modulus $\Omega$ defines a family of distances $d^{\Omega|_n}$ on powers of $X$; we will often abuse notation and write $d^\Omega$ instead when $n$ is clear from the context. Similarly, we will say that a function $f \colon X^n \to \R$ respects $\Omega$ rather than that it respects $\Omega|_n$.

Two examples of weak moduli that will be important for us are the following. The \emph{$1$-Lipschitz weak modulus} $\OmegaL \colon [0, \infty)^\N \to [0, \infty]$ is defined by
\begin{equation}
  \label{eq:OmegaL}
  \OmegaL(\delta) = \sup_i \delta_i, \quad \text{where } \delta = (\delta_0, \delta_1, \ldots).
\end{equation}

The \emph{universal weak modulus for Lipschitz languages} $\OmegaU(\Lip) \colon [0, \infty)^\N \to [0, \infty]$ is defined by
\begin{equation}
  \label{eq:OmegaULip}
  \OmegaU(\Lip)(\delta) = \sum_{i = 0}^\infty i \cdot \delta_i.
\end{equation}
Both of those weak moduli are shift-increasing.

%%%%%%%%%%%%%%%%%%%%%%%%%%%%%%%%%%%%%%%%%%%%%%%%%%

\subsection{Infinitary logic}
\label{sec:infinitary-logic}

Continuous infinitary logic was first introduced by Ben Yaacov and Iovino in \cite{BenYaacov2009a}. The definitions we give here are compatible with theirs.

A \emph{metric language} (or signature) is a collection $L$ of symbols.
For each symbol $s \in L$, the language also determines its \emph{kind} (function or predicate), its arity (a natural number $n_s$), a $n_s$-ary modulus of continuity $\Delta_s$, and, for predicates, a compact interval $I_s \sub \R$ of allowed values for $s$ that we will refer to as a \df{bound}. The language always contains, implicitly, a binary predicate symbol $d$ with $\Delta_d(\delta_1,\delta_2) = \delta_1 + \delta_2$. The bound for $d$ is determined by the language.

An \df{$L$-structure} $A$ is a complete metric space equipped with interpretations of the symbols:
\begin{itemize}
\item Each function symbol $F$ is interpreted by a map $F^A\colon A^{n_F} \rightarrow A$ respecting the modulus $\Delta_F$;
  
\item Each predicate symbol $P$ is interpreted by a function $P^A\colon A^{n_P} \rightarrow \R$ respecting the modulus $\Delta_P$ and bound $I_P$ (i.e., $P^A(\bar a) \in I_P$ for all $\bar a \in A^{n_P}$);
  
\item The symbol $d$ is always interpreted by the distance. It must respect the bound $I_d$.
\end{itemize}

Let $L$ be a metric signature. The logic $\Lomo(L)$ is defined as follows. First, we fix a family $\set{x_i : i \in \N}$ of distinct \emph{variable symbols}. The syntactic objects of the logic are terms and formulas; terms come equipped with a modulus of continuity that they respect and formulas have a modulus of continuity and a bound. 

Terms, atomic formulas, and \df{basic formulas} are constructed inductively as follows.
\begin{itemize}
\item Each $x_i$ is a term that respects the $\bN$-ary modulus $\Delta_{x_i}(\delta) = \delta_i$;
  
\item If $\tau_i$, $i < n$ are terms and $F$ is a function symbol of arity $n$, then $\sigma = F(\bar \tau)$ is a term that respects $\Delta_\sigma = \Delta_F \circ ( \Delta_{\tau_i} : i < n)$;

\item If $P$ a predicate symbol of arity $n$ and $\bar \tau$ are terms, then $\phi = P(\bar \tau)$ is an \emph{atomic formula} that respects the modulus $\Delta_\phi = \Delta_P \circ ( \Delta_{\tau_i} : i < n)$ and the bound $I_P$;

\item If $\set{\phi_i : i < n}$ are atomic formulas with moduli of continuity $\Delta_{\phi_i}$ and bounds $I_{\phi_i}$, and $f\colon \prod_i I_{\phi_i} \rightarrow \R$ is continuous, then $\psi = f(\bar \phi)$ is a \df{basic formula} that respects the modulus $\Delta_\psi = \Delta_f \circ (\Delta_{\phi_i} : i < n)$, where $\Delta_f$ is as per \eqref{eq:ModulusFunction}, and the bound $I_\phi = f(\prod_i I_{\phi_i})$.
\end{itemize}

Next, we define general $\Lomo$-formulas starting from atomic formulas and combining them using finitary connectives, quantifiers, and countable infima and suprema (also called \df{infinitary connectives}). As before, every formula $\phi$ respects some modulus of continuity $\Delta_\phi$ and a bound $I_\phi$. If $\phi$ respects $\Delta$ and $I$, and $\Delta' \geq \Delta$, $I' \supseteq I$, then we will also say that $\phi$ respects $\Delta'$ and $I'$. 
\begin{itemize}
\item Every atomic formula is a formula;
  
\item If $\phi_0, \ldots, \phi_{n-1}$ are formulas that respect $\Delta_{\phi_0}, \ldots, \Delta_{\phi_{n-1}}$ and $I_{\phi_0}, \ldots, I_{\phi_{n-1}}$ and $f \colon \prod_i I_{\phi_i}  \to \R$ is a continuous function, then $f(\phi_0, \ldots, \phi_{n-1})$ is a formula that respects $\Delta_f \circ (\Delta_{\phi_i} : i < n)$ and $f(\prod_i I_{\phi_i})$.

\item If $\phi$ is a formula that respects $\Delta$ and $I$ and $i \in \N$, then $\sup_{x_i} \phi$ and $\inf_{x_i} \phi$ are formulas that respect $\hat \Delta$ and $I$, where
  \begin{equation*}
    \hat \Delta(\delta_0, \ldots, \delta_n) = \Delta(\delta_0, \ldots, \delta_{i-1}, 0, \delta_{i+1}, \ldots, \delta_n).
  \end{equation*}

\item If $\set{\phi_i : i \in \N}$ are formulas, $\Delta$ is a modulus, and $I$ is a bound such that each $\phi_i$ respects $\Delta$ and $I$ then $\bigvee_i \phi_i$, $\bigwedge_i \phi_i$ are also formulas that respect $\Delta$ and $I$. $\bigvee_i \phi_i$ is interpreted as $\sup_i \phi_i$ and $\bigwedge \phi_i$ is interpreted as $\inf_i \phi_i$.
\end{itemize}
The \df{finitary fragment $\Lfin$} is defined as the set of all $\Lomo$-formulas, where the infinitary connectives (the last item above) are not used. Two frequent binary connectives  are $\wedge$ ($\min$) and $\vee$ ($\max$).

Finally, we   need the notion of \df{$(\Omega, I)$-formulas} for some given weak modulus $\Omega$ and bound $I \sub \R$. This definition is more restrictive than that of general formulas in several ways: first, we require that all ($\Omega, I$)-formulas respect $\Omega$ and $I$; second, we only allow $1$-Lipschitz connectives in the inductive definition; and third, we keep track of the variables used and quantifiers are allowed only in a certain order. The last restriction is needed when we compute quantifier ranks. Here, the base of the inductive construction are the basic rather than the atomic formulas; that is, we allow applying an arbitrary connective in the beginning. The formal definition of \df{an $n$-ary $(\Omega, I)$-formula} is by induction as follows.

\begin{itemize}
\item All basic formulas $\phi(x_0, \ldots, x_{n-1})$ that only depend on the first $n$ variables and respect $\Omega$ and $I$ are $n$-ary $(\Omega, I)$-formulas.
  
\item If $\set{\phi_i : i \in \N}$ are $n$-ary $(\Omega, I)$-formulas, then $\bigvee_i \phi_i$ and $\bigwedge_i \phi_i$ are $n$-ary $(\Omega, I)$-formulas.
  
\item If $\phi$ is an $(n+1)$-ary $(\Omega, I)$-formula, then $\inf_{x_n} \, \phi$ and $\sup_{x_n} \, \phi$ are $n$-ary $(\Omega, I)$-formulas.
  
\item If $\phi_0, \ldots, \phi_{k-1}$ are $n$-ary $(\Omega, I)$-formulas and $f \colon \R^k \to \R$ is a $1$-Lipschitz function (for the $\max$ distance on $\R^k$), then $f(\phi_0, \ldots, \phi_{k-1})$ is a $n$-ary $(\Omega, f(I^k))$-formula.
  
\item An $n$-ary $\Omega$-formula is an $n$-ary $(\Omega, I)$-formula for some $I$. An \df{$\Omega$-formula} is an $n$-ary $\Omega$-formula for some $n$. An \df{$\Omega$-sentence} is a $0$-ary $\Omega$-formula.
\end{itemize}

Note that an $n$-ary $(\Omega, I)$-formula automatically respects the modulus $\Omega|_n$ and the bound $I$. Thus the collection of all $n$-ary ($\Omega, I$)-formulas is equicontinuous and uniformly bounded; in particular, we do not need any further equicontinuity and boundedness requirements in the second item of the definition. Note, however, that, as $\Omega$ is not required to be symmetric, our variables are not necessarily interchangeable. If, however, $\Omega$ is symmetric (as is the case with $\OmegaL$), then we can quantify over any variable and not only over the one with the largest index. The condition that $\Omega$ is shift-increasing translates into the fact that we are allowed to substitute variables with bigger indices for free variables in formulas: if $\phi(x_0, \ldots, x_n)$ is an $\Omega$-formula and $i_0 < i_1 < \cdots < i_{n-1}$, then $\phi(x_{i_0}, \ldots, x_{i_n})$ is also an $\Omega$-formula. This property turns out to be very convenient when one tries to write actual formulas.

The notion of an $\Omega$-formula becomes more permissive as $\Omega$ becomes larger; we will see later (Corollary~\ref{c:universal-true}) that for a certain choice of $\Omega$, every $\Lomo$-sentence is equivalent to an $\Omega$-sentence. On the other hand, the expressive power of $\OmegaL$-formulas is strictly weaker than that of full $\Lomo$-logic (see Section~\ref{sec:some-examples}). This distinction is purely an infinitary phenomenon: if one restricts to $\Lfin$, then it follows from \cite{BenYaacov2013b}*{Corollary~1.7} that any formula can be uniformly approximated by a Lipschitz formula and thus the values of $1$-Lipschitz formulas completely determine the values of all $\Lfin$-formulas.

Terms and formulas in $\Lomo(L)$ can naturally be interpreted in any $L$-structure $A$: every term $\sigma$ is interpreted as a function $\sigma^A \colon A^\N \to A$ and every formula $\phi$ is interpreted as a function $\phi^A \colon A^\N \to \R$ that obeys its modulus and bound. If $\phi$ is an $n$-ary formula, then it only depends on the first $n$ variables, so its interpretation can be considered as a function $A^n \to \R$. Sometimes we will write $A \models (\phi = r)$ instead of $\phi^A = r$.

Formulas of fixed arity $n$ are naturally equipped with a seminorm defined as follows:
\begin{equation}
  \label{eq:norm-formulas}
  \nm{\phi} = \sup \set{|\phi(\bar a)| : A \text{ is a structure and } \bar a \in A^n}.
\end{equation}
The norm $\nm{\phi}$ is always finite because an interpretation of a formula is required to obey its bound $I_\phi$. The following basic fact will be needed later.
\begin{lemma}
  \label{l:formulas-dens-char}
  Let $\kappa$ be an infinite cardinal. If the language $L$ has size at most $\kappa$, then the space of $\Lfin$-formulas, equipped with the norm given by \eqref{eq:norm-formulas} has density character at most $\kappa$.
\end{lemma}
\begin{proof}[Proof sketch]
  The only possible problem is that we allow arbitrary continuous functions as connectives and there are uncountably many of them. However, as the space of continuous functions defined on a \emph{compact subset of $\R^n$} is separable in the uniform norm, we can use a countable collection of connectives and thus obtain a dense set of size $\kappa$. (In fact, it is possible to use only finitely many connectives.)
\end{proof}
Note, however, that the space of (even quantifier-free) $\Lomo$-formulas has density character $2^{|L|}$.

\begin{remark}
  Our framework also allows us to treat unbounded predicates. If $P$ is an unbounded predicate, we replace it with an infinite family $\set{P_n : n \in \N}$ of predicate symbols interpreted as $P_n = \min(P, n)$. In the special case where the distance $d$ is unbounded, we take $d_1$ to be the ``official'' distance required by the language. This does not change much as $d$ and $d_1$ are uniformly equivalent. Note also that isomorphism is preserved by this procedure.
\end{remark}

%%%%%%%%%%%%%%%%%%%%%%%%%%%%%%%%%%%%%%%%%%%%%%%%%%

\section{The back-and-forth hierarchy and Scott ranks}
\label{sec:scott-hierarchy}

\subsection{The back-and-forth pseudo-distances}
\label{sec:back-forth-pseudo}

Throughout, we fix a signature $L$ and a weak modulus $\Omega \colon [0,\infty)^\bN \rightarrow [0,\infty]$.

The $r_\alpha$ pseudo-distances that we define in this subsection are the continuous analogue of the back-and-forth equivalence relations for classical structures. Note that $r_\alpha$ take values in $[0, \infty]$.
\begin{defn}
  \label{dfn:ScottDistance}
  Let $\alpha$ be an ordinal or the symbol $\infty$ greater than all ordinals.
  Let $n \in \bN$, let $A$ and $B$ be structures and let $\bar a \in A^n$, $\bar b \in B^n$.
  We define the \emph{back-and-forth pseudo-distance (of rank $\alpha$ and arity $n$, with respect to $\Omega$)}, denoted by $r_{\alpha,n}^{A,B,\Omega}(\bar a,\bar b)$ (or simply $r^{A,B}_\alpha(\bar a,\bar b)$) by induction on $\alpha$ as follows.
  For $\alpha = 0$, we set
  \begin{equation*}
    r_0^{A,B}(\bar a, \bar b) = \sup_\phi \, \bigl| \phi^A(\bar a) - \phi^B(\bar b) \bigr|,
  \end{equation*}
  where $\phi$ varies over all basic $n$-ary $\Omega$-formulas.
  For $\alpha$ limit (or $\infty$),
  \begin{equation*}
    r^{A,B}_\alpha(\bar a, \bar b) = \sup_{\beta<\alpha} \, r^{A,B}_\beta(\bar a, \bar b).
  \end{equation*}
  Finally, for the successor step,
  \begin{equation*}
    r^{A,B}_{\alpha+1}(\bar a, \bar b) = \sup_{c \in A, \, d \in B} \, \inf_{c' \in A, \, d' \in B} \, r^{A,B}_\alpha(\bar ac, \bar bd') \vee r^{A,B}_\alpha(\bar ac', \bar bd).
  \end{equation*}
  We may also write $r_{\alpha,n}(A\bar a,B\bar b)$ instead of $r^{A,B}_{\alpha,n}(\bar a,\bar b)$, allowing $A$ and $B$ to vary together with $\bar a$ and $\bar b$. In case $n = 0$, we write just $r_\alpha(A, B)$.
\end{defn}

For the rest of this section, fix a signature $L$ and a weak modulus $\Omega$.
\begin{lemma} \mbox{}
   
  \label{lem:ScottDistance}
  \begin{enumerate}
  \item \label{i:lem:ScottDistance-1} For fixed $\alpha$ and $n$, $r_{\alpha, n}$ is a pseudo-distance on the class of all pairs $A\bar a$.
    
  \item \label{i:lem:ScottDistance-2} For every $\alpha$, $A$ and $\bar a, \bar b \in A^n$, we have $r^{A,A}_\alpha(\bar a, \bar b) \leq d^{\Omega}(\bar a, \bar b)$.
    
  \item \label{i:lem:ScottDistance-3} For fixed $\alpha$, $n$, $A$, and $B$, the function $r^{A,B}_\alpha$ is uniformly continuous on $A^n \times B^n$, respecting the modulus $\Omega|_n$ on each side. In particular, if $r_\alpha(A \bar a, B \bar b) < \infty$ for some $\bar a \in A^n$, $\bar b \in B^n$, then $r_\alpha(A \bar c, B \bar d) < \infty$ for all $\bar c \in A^n$, $\bar d \in B^n$.
  \end{enumerate}
\end{lemma}
\begin{proof}
  All three items are proved by induction on $\alpha$.
  
  \ref{i:lem:ScottDistance-1} The only non-obvious property is the triangle inequality.
  For $\alpha = 0$ and $\alpha$ limit, this is easy.
  For the successor step, assume that $r_{\alpha+1}(A\bar a,B\bar b) < s$ and $r_{\alpha+1}(B\bar b,C\bar c) < t$ in order to show that $r_{\alpha+1}(A\bar a,C\bar c) \leq s+t$.
  Fix $d \in A$.
  Since $s > r_{\alpha+1}(A\bar a,B\bar b)$, there exists $e \in B$ such that $r_\alpha(A\bar ad,B\bar be) < s$.
  Similarly, there exists $f$ such that $r_\alpha(B\bar be,\bar cf) < t$.
  By the induction hypothesis, $r_\alpha(A\bar ad,C\bar cf) < s+t$.
  Similarly, for all $f$ there exists $d$ such that the same holds, so $r_{\alpha+1}(A\bar a,C\bar c) \leq s+t$ and we are done.

  \ref{i:lem:ScottDistance-2} For $\alpha = 0$, this holds since the interpretation of an $n$-ary $\Omega$-formula respects $\Omega|_n$. For limit steps, this is clear, and at the successor step, take $d' = c$ and $c' = d$ and note that by the definition of $d^\Omega$, $d^\Omega(\bar a c, \bar b c) = d^\Omega(\bar a, \bar b)$.

  \ref{i:lem:ScottDistance-3} follows from \ref{i:lem:ScottDistance-1} and \ref{i:lem:ScottDistance-2}.
\end{proof}

The next lemma shows that the $r_\alpha$ stabilize at a certain point.
\begin{lemma}
  \label{lem:IncreasingScottDistance}
  The following statements hold:
  \begin{enumerate}
  \item \label{item:Increasing} If $\beta < \alpha$ then $r_\beta\leq r_\alpha$ (i.e., $r_{\beta,n} \leq r_{\alpha,n}$ for all $n$);
    
  \item \label{item:IncreasingScottDistanceStabilise}
    If $\kappa$ is an infinite cardinal and $A$ and $B$ are structures of density character at most $\kappa$, then there exists $\alpha < \kappa^+$ such that $r^{A,B}_{\alpha+1} = r^{A,B}_\alpha$.
    Moreover, in this case, the sequence of $r^{A,B}$ stabilizes beyond $\alpha$, i.e., $r^{A,B}_\infty = r^{A,B}_\alpha$.
    
  % \item If $\Omega$ is shift-increasing, then $r_\alpha(\bar ac, \bar bd) \geq r_\alpha(\bar a, \bar b)$. ???? 
  \end{enumerate}
\end{lemma}
\begin{proof} \ref{item:Increasing} We argue by induction on $\alpha$.
  For $\alpha = 0$ and $\alpha$ limit, there is nothing to show.
  We now prove the statement for $\alpha + 1$ assuming that it holds for $\alpha$. By the induction hypothesis, it will suffice to show that $r_\alpha \leq r_{\alpha+1}$, which we do by distinguishing different cases.

  If $\alpha = 0$, then $r_0 \leq r_1$ because a formula $\phi(x_0, \ldots x_{n-1}, x_n)$ that respects $\Omega|_{n+1}$ and does not depend on $x_n$ also respects $\Omega|_n$.
  On the other hand, if $\beta < \alpha$, then $r_\beta \leq r_\alpha$ by the induction hypothesis, so $r_{\beta+1} \leq r_{\alpha+1}$.
  From this, for both $\alpha$ limit and $\alpha$ successor, we deduce that $r_\alpha \leq r_{\alpha+1}$.

  \ref{item:IncreasingScottDistanceStabilise} For $\beta < \kappa^+$, $q \in \Q$, and $n \in \N$, let
  \begin{equation*}
    U_{\beta, q, n} = \set{(\bar a, \bar b) \in A^n \times B^n : r_\beta(\bar a, \bar b) > q}.
  \end{equation*}
  If we keep $q$ and $n$ fixed, $\set{U_{\beta, q, n} : \beta < \kappa^+}$ is an increasing sequence of open sets in the space $A^n \times B^n$ which has weight $\kappa$; therefore the sequence must stabilize at some $\beta(q, n) < \kappa^+$. Finally, set $\alpha = \sup \set{\beta(q, n) : q \in \Q, n \in \N}$.
\end{proof}

The pseudo-distances $r_\alpha$ define naturally equivalence relations $E_\alpha$:
\begin{equation*}
  A \bar a \eqrel{E_\alpha} B \bar b \iff r_\alpha(A \bar a, B \bar b) = 0.
\end{equation*}
In view of Lemma~\ref{lem:IncreasingScottDistance}, we naturally have that $E_\alpha \supseteq E_{\alpha+1}$ and $E_\infty = \bigcap_\alpha E_\alpha$.

If $A$ is a separable structure, say that the sequence $(a_i)_{i \in \N}$ of elements of $A$ is \df{tail-dense} if for every $k$, $\set{a_i : i > k}$ is dense in $A$. It is easy to see that a sequence is tail-dense iff it is dense and it hits every isolated point of $A$ infinitely many times. The following is the key back-and-forth fact that will be used throughout the paper.
\begin{prop}
  \label{p:BackAndForth}
  Suppose that $\Omega$ is shift-increasing. Let $A$ and $B$ be separable structures and let $\bar v \in A^k$, $\bar w \in B^k$. Then we have that $r_\infty(A \bar v, B \bar w) < t$ if and only if there exist tail-dense sequences $a \in A^\bN$ and $b \in B^\bN$ such that $a|_k = \bar v$, $b|_k = \bar w$, and
  \begin{equation*}
    \sup_n r^{A,B}_{0,n}(a|_n, b|_n) < t.
  \end{equation*}
\end{prop}
\begin{proof}
We start with the ``only if'' part.  Let $\cB_A$ be a countable basis for $A$ and let $(U^A_n : n \in \N)$ be a sequence of open sets such that every element of $\cB_A$ appears infinitely often; similarly, define  $(U^B_n : n \in \N)$ for $B$.
  
  We construct the desired sequences by a back-and-forth argument; we only describe the ``forth'' step. Let $\bar a = (a_0, \ldots, a_{n-1})$ and $\bar b = (b_0, \ldots b_{n-1})$ be given (for some even $n \geq k$) and suppose that $r_\infty(A\bar a, B\bar b) < t' < t$. Let $a_n$ be an arbitrary element of $U^A_{n/2}$. We are looking for $b_n \in B$ such that $r_\infty(A\bar a a_n, B\bar b b_n) < t'$. By Lemma~\ref{lem:IncreasingScottDistance}, there exists $\alpha = \alpha_{A, B}$, so that $r_\infty^{A, B} = r_\alpha^{A, B} = r_{\alpha+1}^{A, B}$. We have
  \begin{equation*}
    \begin{split}
      t' &> r_\alpha(A\bar a, B\bar b) \\
      &= r_{\alpha+1}(A\bar a, B\bar b) \\
      &\geq \qsup_{c \in A} \qinf_{d' \in B} r_\alpha(A\bar ac, B\bar bd') \\
      &\geq \qinf_{d' \in B} r_\alpha(A\bar a a_n, B\bar bd').
    \end{split}
  \end{equation*}
  We obtain that there exists $b_n \in B$ such that $r_\alpha(A\bar a a_n, B\bar b b_n) < t'$, which allows us to continue. The fact that $a_{2n} \in U^A_n$ and $b_{2n+1} \in U^B_n$ for every $n$ ensures that both sequences are tail-dense. In the end, we have
  \begin{equation*}
    \sup_n r_0(a|_n, b|_n) \leq \sup_n r_\infty(a|_n, b|_n) \leq t' < t,
  \end{equation*}
  and we are done.

  Conversely, suppose we are given sequences $\bar a$ and $\bar b$ with
  \begin{equation*}
    \sup_n r^{A,B}_{0,n}(\bar a|_n, \bar b|_n) < t' < t.
  \end{equation*}
  We show by induction on $\alpha$ that for any $i_0 < \cdots < i_{n-1}$ and all $\alpha$, we have
  \begin{equation}
    \label{eq:back-forth-1}
    r_\alpha(Aa_{i_0}\ldots a_{i_{n-1}}, Bb_{i_0}\ldots b_{i_{n-1}}) \leq t'.
  \end{equation}
First consider the case that $\alpha=0$ and take some $i_0 < \cdots < i_{n-1}$. Then we have
\begin{equation*}
  r_0(Aa_{i_0}a_{i_1}\ldots a_{i_{n-1}}, Bb_{i_0}b_{i_1} \ldots b_{i_{n-1}}) \leq r_0(Aa_0 a_1 \ldots a_{i_{n-1}}, Bb_0b_1\ldots b_{i_{n-1}}) < t,
\end{equation*}
where the fist inequality follows from the fact that $\Omega$ is shift-increasing and the second from the assumption.

Suppose now that $\alpha=\beta+1$ for some $\beta\geq 0$ for which \eqref{eq:back-forth-1} has been proved. Fix again some $i_0 < \cdots < i_{n-1}$. Then we have
\begin{equation*}
  \begin{split}
    r_{\beta+1}(Aa_{i_0}\ldots a_{i_{n-1}}, Bb_{i_0}\ldots b_{i_{n-1}}) &\leq \limsup_{m \to \infty} r_\beta(Aa_{i_0}\ldots a_{i_{n-1}}a_m, Bb_{i_0}\ldots b_{i_{n-1}}b_m) \\
    & \leq t'.
  \end{split}
\end{equation*}
Indeed, to see why the first inequality holds, suppose that the right-hand side is smaller than $s$. Let $\eps > 0$. Fix $c \in A$ and let $m_k \to \infty$ be such that $a_{m_k} \to c$ and $r_\beta(Aa_{i_0}\ldots a_{i_{n-1}}a_{m_k}, Bb_{i_0}\ldots b_{i_{n-1}}b_{m_k}) < s+\eps$ for all $k$. Now taking $k$ big enough so that $r_\beta(Aa_{i_0}\ldots a_{i_{n-1}}c, Aa_{i_0}\ldots a_{i_{n-1}}a_{m_k}) < \eps$ (which exists because $r_\beta$ is contractive in $d^\Omega$ by Lemma~\ref{lem:ScottDistance}), shows that
\begin{equation*}
  \inf_d r_\beta(Aa_{i_0}\ldots a_{i_{n-1}}a_{m_k}, Bb_{i_0}\ldots b_{i_{n-1}}d) < s+2\eps.
\end{equation*}
The other term in the inductive definition of $r_{\beta+1}$ is treated in a similar way.

The second inequality follows from the inductive hypothesis. The limit case is trivial. This completes the induction and the proof of the proposition.
\end{proof}

Note that the assumption that $\Omega$ is shift-increasing is only used in the ``if'' direction of the proposition.

%%%%%%%%%%%%%%%%%%%%%%%%%%%%%%%%%%%%%%%%%%%%%%%%%%

\subsection{Quantifier rank}
\label{sec:quantifier-rank}

The \df{quantifier rank} of a formula $\phi$, denoted by $\qr \phi$, is defined by induction as follows:
\begin{itemize}
\item $\qr \phi = 0$ if $\phi$ is an atomic formula;
\item $\qr f(\phi_0, \ldots, \phi_{n-1}) = \max_i \qr \phi_i$ if $f$ is a connective;
\item $\qr \bigvee_i \phi_i = \qr \bigwedge_i \phi_i = \sup_i \qr \phi_i$;
\item $\qr (\sup_x \phi) = \qr (\inf_x \phi) = \qr \phi + 1$.
\end{itemize}

The following theorem tells us that, as in the classical case, the distance $r_\alpha$ captures exactly the expressive power of the $\Omega$-formulas of quantifier rank at most $\alpha$.
\begin{theorem}
  \label{th:r_alpha-qr}
  Let $\alpha$ be an ordinal, $A,B \in \cM$, $\bar a \in A^n$ and $\bar b \in B^n$.
  Then
  \begin{equation} \label{eq:ralpha-qr}
    r^{A,B}_\alpha(\bar a, \bar b) = \sup_\phi \, \bigl| \phi^A(\bar a) - \phi^B(\bar b) \bigr|,
  \end{equation}
  where $\phi$ varies over all $n$-ary $\Omega$-formulas of quantifier rank at most $\alpha$.
\end{theorem}
\begin{proof}
  We prove by induction on $\alpha$ that for all bounds $I$,
  \begin{equation*}
    r^{A,B}_\alpha(\bar a, \bar b) \wedge |I| = \sup_\phi \, \bigl| \phi^A(\bar a) - \phi^B(\bar b) \bigr|,
  \end{equation*}
  where $|I|$ denotes the length of $I$ and $\phi$ varies over all $n$-ary $(\Omega, I)$-formulas of quantifier rank at most $\alpha$. For $\alpha = 0$ and limit this is by definition, so assume this for $\alpha$ and let us prove it for $\alpha+1$. For simplicity, suppose that $\min I = 0$.

  Fix $t > 0$, and assume that $r^{A,B}_{\alpha+1}(\bar a,\bar b) > t$.
  Without loss of generality, there exists $c$ such that $r^{A,B}_\alpha(\bar a c,\bar b d) > t$ for all $d$.
  By the induction hypothesis, for each $d$, there exists an $(n+1)$-ary $(\Omega, I)$-formula $\phi_d$ of quantifier rank $\leq \alpha$ such that $\bigl| \phi_d^A(\bar a,c) - \phi_d^B(\bar b,d) \bigr| > t$, and possibly replacing $\phi_d$ with another formula of the same kind, we may assume that $\phi_d^A(\bar a,c) > t > 0 = \phi_d^B(\bar b,d)$.
  Now, $\psi = \sup_{x_n} \, \bigwedge_{d \in \bN} \, \phi_d$ is an $n$-ary $(\Omega, I)$-formula of quantifier rank $\leq \alpha+1$ and
  \begin{equation*}
    \psi^A(\bar a) \geq t > 0 \geq \psi^B(\bar b),  
  \end{equation*}
  which is enough.

  Conversely, assume that $\sup_\phi \, \bigl| \phi^A(\bar a) - \phi^B(\bar b) \bigr| > t$.
  Then $\bigl| \phi^A(\bar a) - \phi^B(\bar b) \bigr| > t$ for some $n$-ary $\Omega$-formula $\phi $ of quantifier rank $\leq \alpha+1$.
  If $\phi$ is of the form $\bigvee_i \, \phi_i$, $\bigwedge_i \, \phi_i$, or $f(\phi_0, \ldots, \phi_{k-1})$, where $f$ is a $1$-Lipschitz connective, then we can replace $\phi$ with one of the $\phi_i$.
  If $\phi$ is basic, use the fact that $r_0 \leq r_{\alpha+1}$ (Lemma~\ref{lem:IncreasingScottDistance}).
  We are left with the case where $\phi = \sup_{x_n} \psi$ (or $\phi = \inf_{x_n} \psi$ but it is similar), where $\qr \psi = \alpha$. We may assume that $\phi^A(\bar a) > t > 0 = \phi^B(\bar b)$.
  In other words, there exists $c$ such that $\psi^A(\bar a,c) > t$ and yet $\psi^B(\bar b,d) \leq 0$ for all $d$.
  By the induction hypothesis, $r^{A,B}_\alpha(\bar a c,\bar b d) > t$ for this one $c$ and all $d$, so $r^{A,B}_{\alpha+1}(\bar a,\bar b) \geq t$ and we are done.
\end{proof}

%%%%%%%%%%%%%%%%%%%%%%%%%%%%%%%%%%%%%%%%%%%%%%%%%%

\subsection{Scott rank and Scott sentence}

In this subsection, given a separable structure $A$, we describe how to associate to it a countable ordinal, its \df{Scott rank}, and construct a sentence that describes it up to $E_\infty$-equivalence.

\begin{defn}
  \label{df:Scott-rank}
  We call the least ordinal $\alpha$ for which $r^{A, B}_\alpha = r^{A, B}_{\alpha+1}$ the \emph{($\Omega$-)Scott rank} of the pair $A,B$, denoted by $\alpha_{A,B}$ (or $\alpha_{\Omega,A,B}$).
  If $A = B$, we call it the \emph{($\Omega$-)Scott rank} of $A$ and denote it by $\alpha_A$.
\end{defn}
Note that by Lemma~\ref{lem:IncreasingScottDistance}, if $A$ is infinite, then $\alpha_A < |A|^+$.

\begin{lemma}
  \label{lem:ScottEquivalence}
  If the structures $A$ and $B$ are $E_\infty$-equivalent, then $\alpha_{A,C} = \alpha_{B,C}$ for any structure $C$ and, in particular, $\alpha_A = \alpha_{A,B} = \alpha_B$.
\end{lemma}
\begin{proof}
  Let $\alpha = \alpha_{A,C}$. By symmetry, it will suffice to prove that $\alpha_{B,C} \leq \alpha$.
  
  Let $\bar b \in B^n$ and $\bar c \in C^n$.
  As $r_\infty(A, B) = 0$, for any $\eps > 0$, there exists a tuple $\bar a \in A^n$ such that
  \begin{equation*}
    r_\alpha(A\bar a,B\bar b) \leq r_{\alpha+1}(A\bar a,B\bar b) \leq r_\infty(A\bar a,B\bar b) < \eps.    
  \end{equation*}
  Since $r_\alpha(A\bar a,C\bar c) = r_{\alpha+1}(A\bar a,C\bar c)$, we have
  \begin{equation*}
    \bigl| r_\alpha(B\bar b, C\bar c) - r_{\alpha+1}(B\bar b, C\bar c) \bigr| <
      \bigl| r_\alpha(A\bar a, C\bar c) - r_{\alpha+1}(A\bar a, C\bar c) \bigr| + 2\eps = 2\eps.
  \end{equation*}
  As $\eps$ is arbitrary, $r_\alpha(B\bar b,C\bar c) = r_{\alpha+1}(B\bar b,C\bar c)$, as desired.
\end{proof}

Next we observe that, analogously to the classical case, for every separable structure $A$, each $\bar a \in A^n$, and each ordinal $\alpha$, there exists a formula $\phi_{\alpha, A\bar a}(\bar x)$ such that for all structures $B$,
\begin{equation} \label{eq:phi-alpha}
  \phi^B_{\alpha,n,A\bar a}(\bar b) = r^{A,B}_{\alpha,n}(\bar a,\bar b) \wedge 1.
\end{equation}
As formulas are always uniformly bounded, taking the minimum with $1$ (or some other constant) above is necessary.

First, we fix a countable, dense subset $D \sub A$. Note that the formulas that we define do depend on this choice of $D$; however for different choices of $D$, one obtains equivalent formulas. For a countable ordinal $\alpha$, $n \in \bN$ and $\bar a \in A^n$, we define inductively the $n$-ary $\Omega$-formula $\phi_{\alpha,n,A\bar a}$ as follows.

For $\alpha = 0$,
\begin{equation*} 
  \phi_{0,n,A\bar a}(x_0, \ldots, x_{n-1}) = \bigvee_\phi \, \bigl| \phi^A(\bar a) - \phi(x_0,\ldots,x_{n-1}) \bigr|,
\end{equation*}
as $\phi$ varies over a countable family of basic $n$-ary $(\Omega, [0, 1])$-formulas, dense in the norm given by \eqref{eq:norm-formulas} (see Lemma~\ref{l:formulas-dens-char}).
For $\alpha$ limit,
\begin{equation*}
  \phi_{\alpha,n,A\bar a} = \bigvee_{\beta<\alpha} \, \phi_{\beta,n,A\bar a}.
\end{equation*}
For a successor,
\begin{equation}
  \label{eq:Scott-formula-a+1}
  \phi_{\alpha+1,n,A\bar a}(x_0, \ldots, x_{n-1}) = \Big( \bigvee_{c \in D} \, \inf_{x_n} \, \phi_{\alpha,n+1,A\bar a c} \Big) \vee \Big( \sup_{x_n} \, \bigwedge_{c \in D} \, \phi_{\alpha,n+1,A\bar a c} \Big).
\end{equation}

An easy induction shows that $\phi_{\alpha,n,A\bar b}$ is an $n$-ary $(\Omega, [0, 1])$-formula of quantifier rank $\alpha$, and that \eqref{eq:phi-alpha} holds.

Now let $\alpha_A$ be the Scott rank of $A$ and note that, as $A$ is separable, by Lemma~\ref{lem:IncreasingScottDistance}, $\alpha_A < \omega_1$. We define $\sigma_A$, the \emph{Scott sentence} of $A$, as
\begin{equation}
  \label{eq:Scott-sentence}
  \sigma_A = \phi_{\alpha_A,0,A} \vee \bigvee_{n, \, \bar a \in D^n} \, \sup_{x_0,\ldots,x_{n-1}} \, \half \bigl| \phi_{\alpha_A,n,A\bar a} - \phi_{\alpha_A+1,n,A\bar a} \bigr|.
\end{equation}
  This is an $(\Omega, [0, 1])$-sentence; the coefficient $\half$ is needed because the function $(x_1, x_2) \mapsto |x_1 - x_2|$ is $2$-Lipschitz and in $\Omega$-formulas, we only allow $1$-Lipschitz connectives.

  The main property of the Scott sentence is the following.
\begin{theorem}
  \label{th:Scott-sentence}
  Let $B$ be a separable structure. Then $B \models (\sigma_A = 0)$ iff $r_\infty(A,B) = 0$.
\end{theorem}
\begin{proof}
  Assume first that $B \models (\sigma_A = 0)$.
  Then the second part of $\sigma_A$ ensures that $r^{A,B}_{\alpha_A} = r^{A,B}_\infty$, and then the first part ensures that $r_{\alpha_A}(A,B) = 0$.
  Together, $r_\infty(A,B) = 0$.

  Conversely, assume that $r_\infty(A,B) = 0$.
  Then $r_{\alpha_A}(A,B) = 0$ so the first part of $\sigma_A$ vanishes on $B$.
  By Lemma~\ref{lem:ScottEquivalence}, we have $\alpha_A = \alpha_{A,B}$, so the second part of $\sigma_A$ also vanishes on $B$.
\end{proof}
% \begin{remark}
%   We would like to point out an important difference between classical and continuous Scott sentences. In classical logic, the Scott sentence associated to (the isomorphism class of) a countable structure is canonical \emph{up to equality} (if one regards the countable conjunctions and disjunctions as unordered), while in our case, the sentence depends on the dense set $D$ chosen in \eqref{eq:Scott-formula-a+1} and is only canonical up to logical equivalence. This is unavoidable: from Hjorth's results on turbulence \cite{Hjorth2000}, it follows that isomorphism of metric models is not classifiable by countable structures and Scott sentences are, of course, countable structures.
% \end{remark}

%%%%%%%%%%%%%%%%%%%%%%%%%%%%%%%%%%%%%%%%%%%%%%%%%%

\section{The space of Polish structures}
\label{sec:space-polish-struct}

From now on  we will assume that the language $L$ is countable,  and we will only consider separable structures. Then it is possible to parametrize all $L$-structures by elements of a Polish space $\cM$, in such a way  that the pseudo-distances $r_\alpha$ become Borel functions on $\cM$.

We will code function symbols by predicates in the following way. If $F$ is an $n$-ary function symbol with modulus $\Delta_F$, we replace it by the $(n+1)$-ary predicate $D_F$ defined by
\begin{equation}
  \label{eq:func-symbol}
  D_F(\bar x, y) = d(F(\bar x), y)
\end{equation}
with modulus of continuity $\Delta_{D_F}$ given by
\begin{equation}
  \label{eq:func-symbol-modulus}
    \Delta_{D_F}(\delta, \delta') = \Delta_F(\delta) + \delta'
\end{equation}
and bound equal to the bound of the metric $d$. Call $L'$ the resulting language.
\begin{lemma}
  \label{l:eliminate-func}
  Every $\Lfin(L)$-formula is equivalent to a $\Lfin(L')$-formula. A similar fact holds  for $\Lomo$.
\end{lemma}
\begin{proof}
  We show by induction that for every term $\tau(\bar x)$, there exists an $L'$-formula $D_\tau(\bar x, y)$ such that $D_\tau(\bar x, y) = d(\tau(\bar x), y)$. Suppose that $\tau = F(\tau_0, \ldots, \tau_{n-1})$, where $F$ is a function symbol and the $\tau_i$ are terms. Define $D_\tau$ by
  \begin{equation*}
    D_\tau(\bar x, y) = \qinf_{\bar z} \set{ D_F(\bar z, y) : \bigvee_i D_{\tau_i}(\bar x, z_i) = 0}.
  \end{equation*}
  This can be written as a legitimate $\Lfin$-formula by \cite{BenYaacov2008}*{Theorem~9.17}.

  Similarly, if $P(\bar \tau_0, \ldots, \tau_{n-1})$ is an atomic formula, the following $L'$-formula is equivalent to it:
  \begin{equation*}
    P'(\bar x) = \qinf_{\bar y} \set{P(\bar y) : \bigvee_i D_{\tau_i}(\bar x, y_i) = 0}.
  \end{equation*}
  Now the lemma follows by induction on formulas.
\end{proof}

Enumerate all predicates in $L'$ as $d = P_0, P_1, P_2, \ldots$ and let $n_0, n_1, \ldots$ be their respective arities. Let $\cM(L) = \cM(L')$ be the set of all $p \in \prod_i \R^{\N^{n_i}}$ such that there exists an $L$-structure $A$ and a tail-dense sequence $(a_i)_{i \in \N}$ of elements of $A$ such that
\begin{equation*}
  p(i)(j_0, \ldots, j_{n_i-1}) = P_i^A(a_{j_0}, \ldots, a_{j_{n_i-1}})
\end{equation*}
for all $i \in \N$, $(j_0, \ldots, j_{n_i-1}) \in \N^{n_i}$. We will also often write just $\cM$ when the language $L$ is clear from the context.
\begin{prop}
  \label{p:MPolish}
  $\cM$ is a $G_\delta$ subset of $\prod_i \R^{\N^{n_i}}$, and therefore a Polish space.
\end{prop}
\begin{proof}
  It is easy to check that $p \in \cM$ iff the following hold: $p(0)$ defines a pseudo-distance on $\N$; the set $\N$ is tail-dense in the metric space $(\N, p(0))$;   for every $i > 0$, $p(i) \colon \N^{n_i} \to \R$ respects the modulus $\Delta_{P_i}$ and the bound $I_{P_i}$ on $\N$;   finally, for every predicate of the form $D_F$, there exists a function $F$ that satisfies \eqref{eq:func-symbol}. Indeed, if $p$ satisfies these conditions, one can just take $A$ to be the completion of $\N$ with respect to the distance $p(0)$, extend all predicates by uniform continuity, and define the functions via \eqref{eq:func-symbol}. The first three of these conditions are clearly $G_\delta$; we check the fourth.

  Suppose that $F$ is a function symbol in $L$ with modulus $\Delta_F$. We claim that a predicate $D_F$ satisfying the modulus \eqref{eq:func-symbol-modulus} comes from a function iff it satisfies the conditions
  \begin{equation}
    \label{eq:MPolish1}
    d(y_1, y_2) \leq D_F(\bar x, y_1) + D_F(\bar x, y_2)
  \end{equation}
  and
  \begin{equation}
    \label{eq:MPolish2}
    \forall \bar x \ \forall \eps > 0 \ \exists y \quad D_F(\bar x, y) < \eps.
  \end{equation}
  The first one is a closed condition and says that $F$ is a function;  the second ensures that $F$ is total. As $D_F$ is uniformly continuous, the quantifiers $\forall \bar x$ and $\exists y$ in \eqref{eq:MPolish2} can be taken to range over the distinguished dense subset, so the condition \eqref{eq:MPolish2} is $G_\delta$. Then any predicate that respects $\Delta_{D_F}$ and satisfies \eqref{eq:MPolish1} and \eqref{eq:MPolish2} is of the form $D_F$ for some function $F$ that respects $\Delta_F$. To see this, fix $\bar x$ and take a sequence $(y_n)$ as given by \eqref{eq:MPolish2} for $\eps = 2^{-n}$; then by \eqref{eq:MPolish1}, the sequence $(y_n)$ is Cauchy and its limit $y$ satisfies $D_F(\bar x, y) = 0$. Define $F(\bar x)$ to be the (unique by \eqref{eq:MPolish1}) $y$ such that $D_F(\bar x, y) = 0$. One then easily checks that $F$ respects $\Delta_F$ and that $D_F(\bar x, y) = d(F(\bar x), y)$.
\end{proof}

We will consider an element $A \in \cM$ as a complete structure with a distinguished tail-dense set $\N \sub A$; thus we will write $P_i^A(a_0, \ldots, a_{n_i-1})$ instead of $A(i)(a_0, \ldots, a_{n_i-1})$ for $a_0, \ldots, a_{n_i-1} \in \N$. In this way, we also interpret  $r_\alpha(A \bar a, B \bar b)$, where $A, B \in \cM$ and $\bar a, \bar b \in \N^n$. Thus $r_{\alpha,n}$ is a pseudo-distance on $\cM \times \bN^n$. Since $\bN$ is dense and $r^{A,B}_\alpha$ is continuous in each variable, we also have
\begin{equation}
  \label{eq:ralpha+1}
  r^{A,B}_{\alpha+1}(\bar a,\bar b) = \sup_{c,d \in \bN} \, \inf_{c',d' \in \bN} \, r^{A,B}_\alpha(\bar ac,\bar bd') \vee r^{A,B}_\alpha(\bar ac',\bar bd),
\end{equation}
that is, it is enough to take suprema and infima over the distinguished dense sets.

\begin{prop} 
  \label{p:Borel}
  The following statements hold:
  \begin{enumerate}
  \item \label{i:p:Borel1} For every formula $\phi(\bar x)$, the function $\cM \times \N^n \to \R$, $A \bar a \mapsto \phi^A(\bar a)$ is Borel.
    
  \item \label{i:p:Borel2} For every $n \in \N$ and $\alpha < \omega_1$, the function $r_{\alpha, n} \colon (\cM \times \N^n)^2 \to \R$ is Borel. 
    
  \item \label{i:p:Borel3}
    If $\Omega$ is shift-increasing, then for every $s \in \R$, the set
  \begin{equation*}
    \set{(A \bar a, B \bar b) \in (\cM \times \N^n)^2 : r_\infty(A \bar a, B \bar b) < s}  
  \end{equation*}
  is analytic.
\end{enumerate}
\end{prop}
\begin{proof}
  \ref{i:p:Borel1} By Lemma~\ref{l:eliminate-func}, it is enough to prove the claim for $L'$-formulas, which is done by induction. Evaluation of $L'$-atomic formulas is a continuous function. For the quantifier step, note that it suffices to quantify over the distinguished dense set.
  
  \ref{i:p:Borel2} For $\alpha = 0$, recall that $r_0(A\bar a, B \bar b) \leq s$ iff for all basic $\phi$ that obey $\Omega$, $|\phi^A(\bar a) - \phi^B(\bar b)| \leq s$. By \ref{i:p:Borel1}, all of those are Borel conditions and by Lemma~\ref{l:formulas-dens-char}, it suffices to consider only countably many of them. The rest of the proof is a straightforward induction on $\alpha$ using \eqref{eq:ralpha+1}.

  \ref{i:p:Borel3} This follows from Proposition~\ref{p:BackAndForth}, which gives an analytic description of the condition $r_\infty(A \bar a, B \bar b) < s$.
\end{proof}

We obtain the following corollary from Theorem~\ref{th:Scott-sentence}.
\begin{cor}
  \label{c:classes-Borel}
  Let $L$ be a countable signature and $\Omega$ be a weak modulus. Then for every separable $L$-structure $A$, the set
  \begin{equation*}
    \set{B \in \cM(L) : r_\infty^\Omega(A, B) = 0}
  \end{equation*}
  is Borel.
\end{cor}

%%%%%%%%%%%%%%%%%%%%%%%%%%%%%%%%%%%%%%%%%%%%%%%%%%

\section{The universal weak modulus}
\label{sec:universal-case}

In this section, we show that for every countable language $L$, there exists a \df{universal weak modulus} $\Omega$ such that every $\Lomo(L)$-sentence is equivalent to an $\Omega$-sentence (Corollary~\ref{c:universal-true}) and for which the equivalence relation $E_\infty$ is the finest one possible, that of isomorphism (Theorem~\ref{th:UniversalRInfty}).

What we will need from the universal modulus is that $\Omega$-formulas be sufficiently expressive; see Proposition~\ref{p:universal-mod} below. The definition is chosen in such a way to make this work.
\begin{defn}
  \label{df:universal-mod}
  Let $L$ be a countable signature. We   say that a weak modulus $\Omega$ is \df{universal for $L$} if it satisfies the following conditions:
  \begin{enumerate}  
  \item \label{i:um:1} For every atomic formula $\phi(x_0, \ldots, x_{k-1})$, there exists $n$ such that
    \begin{equation*}
      \Delta_\phi(\delta_0, \ldots, \delta_{k-1}) \leq \Omega|_n(0, \ldots, 0, \delta_0, \ldots, \delta_{k-1});
    \end{equation*}
    
  \item \label{i:um:2} For every $k \in \N$ and every $M > 0$, there exists $n$ such that
    \begin{equation*}
      M \cdot \Omega|_k(\delta_0, \ldots, \delta_{k-1}) \leq \Omega|_n(0, \ldots, 0, \delta_0, \ldots, \delta_{k-1});
    \end{equation*}
    
  \item \label{i:um:3} For every $k, n \in \N$,
    \begin{equation*}
      \Omega|_k(\delta_0, \ldots, \delta_{k-1}) + \Omega|_n(\gamma_0, \ldots, \gamma_{n-1}) \leq \Omega|_{k+n}(\delta_0, \ldots, \delta_{k-1}, \gamma_0, \ldots, \gamma_{n-1});
    \end{equation*}
    
  \item \label{i:um:4} $\Omega$ is shift-increasing.
  \end{enumerate}
\end{defn}

\begin{prop}
  \label{p:universal-mod}
  Let $\Omega$ be a universal weak modulus for $L$. Then the following hold:
  \begin{enumerate}
  \item \label{i:pum:1} For every $k$-ary atomic formula $\phi(\bar x)$, there exists $n$ such that $\phi(x_n, \ldots, x_{n+k-1})$ is an $\Omega$-formula;

  \item \label{i:pum:2} For every $k$-ary $\Omega$-formula $\phi(\bar x)$ and every $M > 0$, there exists $n$ such that $M \cdot \phi(x_n, \ldots, x_{n+k-1})$ is an $\Omega$-formula;

  \item \label{i:pum:3} For all tuples
    \begin{equation*}
      i_0 < \cdots < i_{n-1} < j_0 < \cdots < j_{n-1},  
    \end{equation*}
    $d^\Omega\big((x_{i_0}, \ldots, x_{i_{n-1}}), (x_{j_0}, \ldots, x_{j_{n-1}})\big)$ is an $\Omega$-formula.
  \end{enumerate}
\end{prop}
\begin{proof}
  The items \ref{i:pum:1} and \ref{i:pum:2} follow from the corresponding ones in Definition~\ref{df:universal-mod}. We check \ref{i:pum:3}. Let $A$ be an $L$-structure and $\bar a, \bar b, \bar c, \bar d \in A^n$. Let $\delta_i = d(a_i, c_i)$ and $\gamma_i = d(b_i, d_i)$. We have:
  \begin{equation*}
    \begin{split}
      |d^\Omega(\bar a, \bar b) - d^\Omega(\bar c, \bar d)| &\leq d^\Omega(\bar a, \bar c) + d^\Omega(\bar b, \bar d) \\
      &= \Omega(\delta_0, \ldots, \delta_{n-1}) + \Omega(\gamma_0, \ldots, \gamma_{n-1}) \\
      &\leq \Omega(\delta_0, \ldots, \delta_{n-1}, \gamma_0, \ldots, \gamma_{n-1}) \\
      &\leq \Omega(\ldots, \delta_0, \ldots, \delta_{n-1}, \ldots, \gamma_0, \ldots, \gamma_{n-1}, \ldots),
    \end{split}
  \end{equation*}
  where in the last line, $\delta_0, \ldots, \delta_{n-1}, \gamma_0, \ldots, \gamma_{n-1}$ are in positions $i_0, \ldots, i_{n-1}, j_0, \ldots, j_{n-1}$, respectively, and the other positions are filled with zeros. The second inequality is condition \ref{i:um:3} in Definition~\ref{df:universal-mod} and the last one follows from the fact that $\Omega$ is shift-increasing. This completes the proof of \ref{i:pum:3}.
\end{proof}

\begin{prop}
  \label{p:univ-mod-exists}
  For every signature $L$, a universal modulus $\OmegaU(L)$ for $L$ exists. If $L$ is a \df{Lipschitz language} (that is, all moduli of continuity for symbols in $L$ are linear functions), then we can take $\Omega = \OmegaU(\Lip)$ as defined by \eqref{eq:OmegaULip}.
\end{prop}
\begin{proof}
  Let $\set{\phi_i}_{i \in \N}$ be an enumeration of all atomic formulas in $L$. Let
  \begin{equation*}
    \OmegaU(L)(\delta_0, \delta_1, \ldots) = \sum_{i = 0}^\infty i \cdot \sup_{k \leq i} \Delta_{\phi_k}(\delta_i, \ldots, \delta_i).
  \end{equation*}
  One easily checks that all conditions in Definition~\ref{df:universal-mod} are satisfied. Similarly for $\OmegaU(\Lip)$ and Lipschitz languages.
\end{proof}

\begin{remark}
  Our definition of a universal weak modulus is somewhat arbitrary: we have put together all conditions that we need (here, as well as in Sections \ref{sec:lopez-escob-theor} and \ref{sec:bounded-rank}) and what is important for us is the fact that such a universal modulus exists. For other purposes, one might need additional properties.
\end{remark}

\begin{theorem}
  \label{th:UniversalRInfty}
  Let $L$ be a countable signature and let $\Omega$ be a universal weak modulus for $L$.
  Let $A$ and $B$ be separable $L$-structures and $\bar v \in A^k$, $\bar w \in B^k$. Then
  \begin{equation*}
    r_\infty^\Omega(A \bar v, B \bar w) = \inf \set{d^\Omega(f(\bar v), \bar w) : f \text{ is an isomorphism } A \to B}.
  \end{equation*}
  In particular,
  \begin{equation*}
    r_\infty^\Omega(A, B) = \begin{cases}
      0 & \text{if } A \cong B, \\
      \infty & \text{otherwise.}
    \end{cases}
  \end{equation*}
\end{theorem}
\begin{proof}
  The inequality $\leq$ is clear: if $f \colon A \to B$ is an isomorphism, then it follows from Lemma~\ref{lem:ScottDistance} that for all $\alpha$,
  \begin{equation*}
    r_\alpha(A \bar v, B \bar w) \leq d^\Omega(f(\bar v), \bar w).
  \end{equation*}

  For the   inequality $\ge$, suppose that $r_\infty(A \bar v, B \bar w) < t$.
  By Proposition~\ref{p:BackAndForth}, there exist tail-dense sequences $a \in A^\bN$ and $b \in B^\bN$ such that $a|_k = \bar v$, $b|_k = \bar w$ and for all $n$, $r^{A,B}_0(a|_n, b|_n) \leq t$.

  Consider an $n$-ary atomic formula $\phi(\bar y)$. By Proposition~\ref{p:universal-mod}  \ref{i:pum:1}, \ref{i:pum:2} and the fact that $\Omega$ is shift-increasing, for every $M > 0$, there exists $l(M, \phi)$ such that for $\sigma \in \bN^n$ with $\min \sigma > l$, 
\begin{equation*}
  \phi_\sigma =  M \cdot \phi(x_{\sigma_0},\ldots,x_{\sigma_{n-1}})
\end{equation*}
is a basic $\Omega$-formula. It follows that
  \begin{equation}
    \label{eq:phi-sigma-small}
  \bigl| \phi_\sigma(a_{\sigma_0},\ldots,a_{\sigma_{n-1}}) - \phi_\sigma(b_{\sigma_0},\ldots,b_{\sigma_{n-1}}) \bigr|
  \leq r_0(a|_{\max \sigma}, b|_{\max \sigma}) < t,
\end{equation}
i.e., for all $\sigma$ with $\min \sigma > l(M, \phi)$,
\begin{equation*}
\bigl| \phi(a_{\sigma_0},\ldots,a_{\sigma_{n-1}}) - \phi(b_{\sigma_0},\ldots,b_{\sigma_{n-1}}) \bigr| < t/M,
\end{equation*}
which implies that
\begin{equation}
  \label{eq:UniversalDeltaLipschitzFormula}
  \bigl| \phi(a_{\sigma_0},\ldots,a_{\sigma_{n-1}}) - \phi(b_{\sigma_0},\ldots,b_{\sigma_{n-1}}) \bigr| \rightarrow 0 \text{ as } \min \sigma \rightarrow \infty.
\end{equation}

Applied to the formula $d(y_0,y_1)$, this means that for a strictly increasing sequence $(m_n) \subseteq \bN$, if one of the subsequences $(a_{m_n})$ and $(b_{m_n})$ is Cauchy, then so is the other. If $\lim a_{m_n} = \lim a_{m'_n}$, then $(a_{m_n}, a_{m'_n})$ is Cauchy, therefore so is $(b_{m_n}, b_{m'_n})$, and thus $\lim b_{m_n} = \lim b_{m'_n}$ (and vice versa). We therefore obtain a bijection $\theta\colon A \to B$ defined by $\theta\bigl( \lim a_{m_n} \bigr) = \lim b_{m_n}$.
  By \eqref{eq:UniversalDeltaLipschitzFormula} again, $\theta$ is an isomorphism. Next we check that $d^\Omega(\theta(a|_k), b|_k) \leq t$. For $i < k$, let $(m_{i, n})_n$ be strictly increasing sequences such that $a_{m_{i, n}} \to a_i$ as $n \to \infty$ and
  \begin{equation*}
    k < m_{0, n} < \cdots < m_{k-1, n} \quad \text{ for each } n.
  \end{equation*}
  Then Proposition~\ref{p:universal-mod} \ref{i:pum:3} and \eqref{eq:phi-sigma-small} give us that
  \begin{equation*}
    \big|d^\Omega\big(a|_k, (a_{m_{0, n}}, \ldots, a_{m_{k-1, n}}) \big)
    - d^\Omega\big(b|_k, (b_{m_{0, n}}, \ldots, b_{m_{k-1, n}}) \big) \big| < t
  \end{equation*}
  for all $n$ and taking limits yields $d^\Omega(b|_k, \theta(a|_k)) \leq t$ as desired.
\end{proof}

Applying Corollary~\ref{c:classes-Borel} allows us to recover one of the original applications of Scott sentences.
\begin{cor}
  \label{c:iso-classes-Borel}
  Let $L$ be a countable signature and $A$ be a separable $L$-structure. Then the set $\set{B \in \cM(L) : B \cong A}$ is Borel.
\end{cor}
\begin{remark}
We should note that Corollary~\ref{c:iso-classes-Borel} also follows from the main result of \cite{Elliott2013}, where the authors prove that isomorphism of Polish metric structures is Borel reducible to a Polish group action (which implies that classes are Borel). However, our result is more precise as it gives a bound of the Borel complexity of the isomorphism class in terms of the Scott rank.
\end{remark}

\begin{remark}
  \label{rem:non-closed-orbits}
  Note that in Theorem~\ref{th:UniversalRInfty}, if $r_\infty(A\bar a, B \bar b) = 0$, this does \emph{not} imply that there is an isomorphism between $A$ and $B$ that sends $\bar a$ to $\bar b$. An easy example of this is when $A = B$ and the orbits of the action of $\Aut(A)$ on $A$ are not closed. Then, for any $a$ and $b$ in the same orbit closure of $\Aut(A)$, $r_\infty(Aa, Ab) = 0$. However, if $a$ and $b$ are not in the same orbit, there is no automorphism that maps $a$ to $b$. Such a structure can be obtained as follows: if $G$ is a Polish group, its   completion under a left-invariant metric $G_L$ can be made into a metric structure in such a way that $\Aut(G_L) = G$ with the action of $G$ on $G_L$ by left translation. If $G \neq G_L$ (this holds, e.g., for $G= S_\infty$), this provides an example of an automorphism group with  orbits that are not closed. See Melleray~\cite{Melleray2010} for more details.
\end{remark}

\begin{remark}
  It is possible to take a slightly different approach  that may be more suitable for some purposes when defining the back-and-forth pseudo-distances. In the definition of $r_0$, one could replace basic formulas by full $\Lfin$-formulas that obey $\Omega$. This has the advantage of being much more robust with respect to syntactical considerations: one can assume from the start, without loss of generality, that the signature contains no function symbols and that all predicates are Lipschitz (see \cite{BenYaacov2013b}*{Corollary~1.7}), thus allowing a unique universal weak modulus that works for all languages (namely, $\OmegaU(\Lip)$). The main disadvantage of this approach is that $r_0$ becomes very hard to compute, while with our definition, computations are sometimes feasible (see Section~\ref{sec:some-examples}).
\end{remark}

%%%%%%%%%%%%%%%%%%%%%%%%%%%%%%%%%%%%%%%%%%%%%%%%%%

\section{A López-Escobar theorem}
\label{sec:lopez-escob-theor}

Next we prove a continuous analogue of the classical López-Escobar theorem: that every Borel set of structures invariant under isomorphism is the set of models of some $\Lomo$-sentence. This is a converse to Proposition~\ref{p:Borel} \ref{i:p:Borel1}.

Let  $L$ be a fixed countable language, and, as before, denote by $\cM$ the space of Polish metric $L$-structures. Let $\Omega$ be a universal weak modulus for $L$. Let $\cong$ be the (analytic) equivalence relation of isomorphism on $\cM$ (by Theorem~\ref{th:UniversalRInfty}, this is the same as $E_\infty$).

We proceed to the main theorem of this section. We use the definition of Baire class for real valued, Borel functions given in \cite{Kechris1995}*{24.1}. If $\theta$ is the least ordinal such that the function $U$ is of Baire class $\theta$, we write $\BC(U) =\theta$.
\begin{theorem}
  \label{th:Lopez-Escobar}
  Let $I \sub \R$ be a compact interval and $U \colon \cM \to I$ be a Borel function that is $\cong$-invariant. Then there exists an $(\Omega, I)$-sentence $\phi$ such that
  \begin{equation*}
    U(A) = \phi^A \quad \text{for all } A \in \cM.
  \end{equation*}
  Moreover, we have that $\qr(\phi) \le \omega \cdot 2 \cdot (2 + \BC(U) +1)$. 
\end{theorem}
There are two known proofs of the López-Escobar theorem in the classical case: the original one in \cite{Lopez-Escobar1965}, based on proof theory, and another one by Vaught~\cite{Vaught1974}, based on Baire category and the fact that isomorphism is given by an action of $S_\infty$. It is the latter that we adapt to our situation. In our setting, we do not have a group action around but we do have Baire category and it will turn out that this is sufficient.

If $A \in \cM$ is a model, we denote by $[A]$ the $\cong$-equivalence class of $A$ in $\cM$ (by Corollary~\ref{c:iso-classes-Borel}, this is a Borel set). We also define
\[
D(A) = \set[\big]{y \in A^\N : \set{y(n) : n \in \N} \text{ is tail-dense in } A}.
\]
Here and later, by a slight abuse of notation, we denote by $A$ both the element of $\cM$ (which prescribes the values of all predicates on a dense set) and the actual model (the completion of $\N$ with respect to the metric). Note that $D(A)$ is a $G_\delta$ set in $A^\N$, and therefore a Polish space. %(see, e.g, \cite{Kechris1995}*{Theorem~3.11}).

For each $A \in \cM$, there is a natural continuous surjection $\pi_A \colon D(A) \to [A]$ defined by
\[
P^{\pi_A(y)}(i_0, \ldots, i_{n-1}) = P^A(y(i_0), \ldots, y(i_{n-1}))
\]
for all predicates $P \in L$ of arity $n$ and all $i_0, \ldots, i_{n-1} \in \N$. Because of the way we code models, we may assume, without loss of generality, that the language $L$ does not contain function symbols. The map $\pi_A$ will allow us to push forward the ideal of meager sets on $D(A)$ to $[A]$, which is an essential element of the proof.

If $X$ is a Baire topological space, we will use the category quantifiers ``$\exists^*$'' to mean ``for a non-meagre set of'' and ``$\forall^*$'' to mean ``for a comeagre set of''. If $f \colon X \to \R$ is a Baire measurable function, define the operators $\supstar$ (essential supremum) and $\infstar$ (essential infimum) as follows:
\[ \begin{split}
  \supstar_{x \in X} f(x) > t &\iff \exists^* x \in X \ f(x) > t \\
  \infstar_{x \in X} f(x) < t &\iff \exists^* x \in X \ f(x) < t.
\end{split} \]
Note that if $f$ is continuous, then
\begin{equation*}
\supstar_{x \in X} f(x) = \sup_{x \in X} f(x) \quad \And \quad \infstar_{x \in X} f(x) = \inf_{x \in X} f(x).
\end{equation*}

If $y$ and $z$ are finite or infinite sequences of elements of a metric space $(Y, d)$  at least one of which is finite, we will abuse notation and write $d^\Omega(y, z)$ instead of $d^\Omega(y|_m, z|_m)$, where $m = \min(|y|, |z|)$. If $y$ or $z$ is the empty sequence, we set $d^\Omega(y, z) = 0$.

We will call a function $F \colon \cM \to \R$ a \df{basic continuous function} if there exists $k \in \N$, predicates $P_i$ of arity $n_i$ for $i < k$, elements $a_{i, j} \in \N$ ($i < k, j < n_i$), and a Lipschitz (for the $\max$ distance on $\R^k$) function $f \colon \R^k \to \R$ such that
\begin{equation*}
  F(A) = f\big( (P_i^A(a_{i, 0}, \ldots, a_{i, n_i-1}))_{i<k} \big).
\end{equation*}
\begin{lemma}
  \label{l:Borel-functions}
  Let $I \sub \R$ be a compact interval. The class of Borel functions $\cM \to I$ is the smallest class that contains the basic continuous functions taking values in $I$ and is closed under countable suprema and infima.
\end{lemma}
\begin{proof}
  First recall that, by \cite{Kechris1995}*{11.6}, the class of Borel functions on a Polish space is the smallest class containing all continuous functions and closed under pointwise limits. If $\lim_n f_n$ exists then  $\lim_n f_n = \inf_n \sup_{m \geq n} f_m$, so it is also the smallest class containing all continuous functions and closed under countable suprema and infima.

  Thus it suffices to prove that all continuous functions can be obtained from the basic ones using pointwise limits. As $\cM \sub \prod_i I_{P_i}^{\N^{n_{P_i}}}$ is a $G_\delta$ set (Proposition~\ref{p:MPolish}), any continuous function on $\cM$ extends to a Baire class $2$ function on the compact space $Z = \prod_i I_{P_i}^{\N^{n_{P_i}}}$. Any continuous function on $Z$ is uniformly continuous and can be approximated by a function that only depends on finitely many coordinates, that is, a function of the type $A \mapsto f\big( (P_i^A(a_{i, 0}, \ldots, a_{i, n_i-1}))_{i < k} \big)$ with $f \colon K \to \R$ continuous, where $K = \prod_i I_{P_i}$. By Stone--Weierstrass, $f$ can be approximated by Lipschitz functions, and such functions are permitted in our definition of basic continuous functions. By taking limits (two times), we can obtain any Baire class $2$ function on $Z$ from continuous functions.
\end{proof}

If $M > 0$, $U \colon \cM \to [0, M]$ is a Borel function, and $k \in \N$, define the function $U^{*k} \colon \cM \times \N^k \to [0, M]$ as follows:
\[
U^{*k}(A, \bar u) = \infstar_{y \in D(A)} U(\pi_A(y)) \vee k d^\Omega(y, \bar u).
\]

The following result easily implies Theorem~\ref{th:Lopez-Escobar} and is better suited for an inductive argument. 
\begin{theorem}
  \label{th:Vaught}
  Let $M > 0$. For every Borel function $U \colon \cM \to [0, M]$ and for every $k \in \N$, there exists $l \in \N$ and an ($\Omega, [0, M])$-formula $\phi_{U,k}(x_l, \ldots, x_{l+k-1})$ such that
  \begin{equation}
    \label{eq:LE-techn}
    U^{*k}(A, \bar u) = \phi_{U,k}^A(u_0, \ldots, u_{k-1})    
  \end{equation}
for all $(A, \bar u) \in \cM  \times \N^k$.
Moreover, \begin{equation} \label{eqn:qrr}  \qr(\phi_{U,k}) \le \omega \cdot 2 \cdot (2 + \BC(U) +1). \end{equation}
\end{theorem}
\begin{proof}
  For the main statement of the theorem, by Lemma~\ref{l:Borel-functions}, it is enough to check that the class of functions $U$ that satisfy the theorem contains the basic continuous functions and is closed under countable infima and the operation $U \mapsto M - U$. (This operation exchanges infima and suprema and preserves the interval $[0, M]$.) Thereafter we will check ``moreover'' statement  by  bounding  the quantifier rank of the constructed formulas.

First suppose that $U$ is a basic continuous function. Then, by the properties of the universal modulus (Proposition~\ref{p:universal-mod} \ref{i:pum:1}, \ref{i:pum:2} and the fact that $\Omega$ is shift-increasing), there exist $n \in \N$, variables $z_0, \ldots, z_{n-1}$, and a basic $\Omega$-formula $\theta(z_0, \ldots, z_{n-1})$ such that
  $U(A) = \theta^A(0, \ldots, n-1)$. As $U$ is continuous, $\sup$ and $\supstar$ coincide, and by the properties of the universal modulus again, there exists $l \in \N$ such that 
\begin{equation} \label{eqn:basic}
  \phi_{U, k}(x_l, \ldots, x_{l+k-1}) = \qinf_{\bar z} \theta(\bar z) \vee k  d^\Omega(\bar x, \bar z)
\end{equation}
is an $\Omega$-formula and satisfies \eqref{eq:LE-techn}.  (Note that $\phi_{U, k}$ is technically not an $(\Omega, [0, M])$-formula as the subformula $k  d^\Omega(\bar x, \bar z)$ does not respect the bound $[0, M]$. This can be easily fixed by replacing it by $(k  d^\Omega(\bar x, \bar z)) \wedge M$ but we will not do this in order to avoid cluttering the exposition. This change is completely harmless as it transforms $\phi_{U, k}$ into an equivalent formula. A similar remark also applies to the constructions below.)

  Suppose now that $U = \inf_n U_n$ and that the formulas $\phi_{{U_n}, k}$ have already been constructed. We verify that
\begin{equation*}
  \phi_{U, k}(\bar x) = \bigwedge_n \phi_{U_n, k}(\bar x)  
\end{equation*}
works. Using the inductive hypothesis and the fact that $\infstar$ commutes with  taking countable infima, we obtain
\begin{equation*}
  \begin{split}
U^{*k}(A, \bar u) & = \infstar_{y \in D(A)} \inf_n \ U_n (\pi_A(y)) \vee k d^\Omega(y, \bar u) \\ 
		& = \inf_n \infstar_{y \in D(A)} \ U_n (\pi_A(y)) \vee k d^\Omega(y, \bar u) \\
		& = \inf_n \ \phi_{U_n,k}^A(\bar u).
  \end{split}
\end{equation*}

Finally, suppose that $U = M - V$, and that formulas $\phi_{V, m}$ satisfying 
\begin{equation}
  \label{eq:LE:ind-hyp}
  \phi_{V, m}^A(\bar z) = \infstar_{v \in D(A)}  V(v) \vee m d^\Omega(v, \bar z)
\end{equation}
have already been constructed for every $m$.

We show that
\begin{equation*}
  \phi_{U,k}(\bar x) = \bigwedge_m \inf_{z_0, \ldots, z_{m-1}} (M - \phi_{V, m}(\bar z)) \vee k d^\Omega(\bar x, \bar z)
\end{equation*}
satisfies \eqref{eq:LE-techn}. Note also that the variables can be chosen in a such a way that $\phi_{U, k}$ is an $\Omega$-formula.

Fix $A$, $\bar u$, and $r \in \R$  in order to show that
\begin{equation}
  \label{eq:LE:goal}
  U^{*k}(A, \bar u) < r \iff \phi_{U,k}^A(\bar u) < r.
\end{equation}

If $\bar z \in A^m$ and $r > 0$, let $\cB(\bar z, r)$ be the open set
\begin{equation*}
  \cB(\bar z, r) = \set{y \in D(A) : d^\Omega(\bar z, y) < r}.
\end{equation*}
Suppressing $\pi_A$ from the notation, we have that
\begin{multline*}
    U^{*k}(A, \bar u) < r \iff \exists^* y \in D(A) \  \Big( k d^\Omega(\bar u, y) < r \And (M - V(y)) < r \Big) \\
    \iff \exists \text{ open } W \sub D(A) \ \Big( W \sub \cB(\bar u, r/k) \And \forall^*y \in W \ V(y) > M-r \Big).
\end{multline*}
On the other hand,
\begin{align*}
  \phi_{U, k}^A(\bar u) < r
  \iff{}& \exists m, \bar z \ \Big( \big((M - \infstar_{v \in D(A)} V(v)) \vee m d^\Omega(\bar z, v) \big) < r \\
    &\qquad \And k d^\Omega(\bar u, \bar z) < r \Big) \\
  \iff{}& \exists m, \bar z\ \Big( \big( \forall^* v \in D(A) \ V(v) > M - r \Or m d^\Omega(v, \bar z) > M-r \big) \Big) \\
    &\qquad \And k d^\Omega(\bar u, \bar z) < r \Big) \\
  \iff{}& \exists m, \bar z\ \Big( \big( \forall^*v \in \cB(\bar z, (M-r)/m) \ V(v) > M - r \big) \\
  &\qquad \And k d^\Omega(\bar u, \bar z) < r \Big).
\end{align*}
For the direction $(\Leftarrow)$ of \eqref{eq:LE:goal}, suppose that $z$ and $m$ are given. By enlarging $m$ if necessary (and prolonging $\bar z$ arbitrarily), we may assume that $(M-r)/m < r/k - d^\Omega(\bar u, \bar z)$. Then it suffices to take $W = \cB(\bar z, (M-r)/m)$ to witness that $U^{*k}(A, \bar u) < r$.

For the other direction, suppose that $W$ is given. Let $y \in W$ be arbitrary and take $m$ so big that $\cB(y|_m, (M-r)/m) \sub W$. Finally, set $z = y|_m$.

This completes the induction.

Now we give bounds on the quantifier ranks.   Let $F $ be the function defined   by $F(\theta) = \omega \cdot 2 \cdot (2+\theta+1)$ for each ordinal $\theta$. 
 Note that if $U = \lim_n U_n$ for Borel functions $U_n$, then we may assume that the range of each $U_n$ is contained in $[0,M]$ without increasing the Baire class, and we have 
\begin{equation} \label{eqn:Ulimit} U = \limsup_n U_n = \inf_m (M- \inf_{n \ge m} (M- U_n)). \end{equation}
 If $U$ is basic continuous, for $\phi_{U,k}$ defined in \eqref{eqn:basic}, we have $\qr (\phi_{U, k}) < \omega$. 
 
  For the general case, we induct on $\BC(U)$. 
  First suppose that  $\BC(U)=0$, that is, $U$ is continuous. By the proof of Lemma~\ref{l:Borel-functions}, we need to take limits two times to obtain $U$ from basic continuous functions. Moreover, each basic continuous function is represented by a formula of quantifier rank $<\omega$; thus \eqref{eqn:Ulimit} shows that $\qr \phi_{U, k} \leq \omega + \omega \cdot 2 + \omega \cdot 2 = \omega \cdot 5$, whence \eqref{eqn:qrr} holds for $U$.
 
 Now suppose that $\BC(U) = \theta>0$. Then $U = \lim_n U_n$ where $\BC(U_n)< \theta$ for each $n$. By the inductive hypothesis, $\qr(\phi_{U_n, k}) \le F( \BC(U_n))$, so there is $\rho < F(\theta)$ such that for each $n, k$ we have $\qr(\phi_{M-U_n, k}) \le \rho$. Therefore $\qr  (\psi_{m,k}) \le \rho$ where $\psi_{m,k} = \phi_{\inf_{n \ge m}  M-U_n,k}$. A similar argument now shows that $\qr (\phi_{U,k}) \le \rho + \omega \le F(\theta)$. 
\end{proof} 

\begin{proof}[Proof of Theorem~\ref{th:Lopez-Escobar}]
 Suppose $U \colon \cM \to I$ is a Borel function invariant under isomorphism. Let $M = |I|$ and let $U' = U - \min I$, so that $U'$ takes values in $[0, M]$. In Theorem~\ref{th:Vaught}, take $\bar u = \emptyset$ and observe that $U'^{*0}(A, \emptyset) = U'(A)$ for every $A$. Thus,  $\phi= \phi_{U',0} + M $ is (equivalent to) an $(\Omega, I)$-sentence such that $U(A) = \phi^A$ for every $A$. Moreover, such a  sentence has quantifier rank at most $\omega \cdot 2 \cdot (2 + \BC(U)+1)$.
\end{proof}

A standard corollary of the López-Escobar theorem is the Craig interpolation theorem for $\Lomo$-logic. (In fact, López-Escobar first proved the interpolation result and then deduced his theorem from it.) Here we note the continuous version.
\begin{cor}[Interpolation]
  \label{c:interpolation}
  Suppose that  $L_1$ and $L_2$ are two countable signatures and $\phi_1$ and $\phi_2$ are $\Lomo(L_1)$ and $\Lomo(L_2)$ sentences respectively. Suppose that $\phi^A_1 \leq \phi^A_2$ for every separable model $A$ of $\Lomo(L_1 \cup L_2)$. Then there is an interpolating sentence $\theta$ in $\Lomo(L_1 \cap L_2)$ such that $\models \phi_1 \leq \theta \leq \phi_2$.
\end{cor}
\begin{proof}
  Let $L_0 = L_1 \cap L_2$ and let $\pi_1 \colon \cM(L_1) \to \cM(L_0)$ and $\pi_2 \colon \cM(L_2) \to \cM(L_0)$ be the natural restriction maps. For every $r \in \Q$, consider the two analytic sets
  \[ \begin{split}
    \set{A \in \cM(L_0) &: \exists B \in \cM(L_1) \ \pi_1(B) = A \And \phi_1^B > r} \text{ and} \\
    \set{A \in \cM(L_0) &: \exists B \in \cM(L_2) \ \pi_2(B) = A \And \phi_2^B < r}.   
  \end{split} \]
  They are $\cong$-invariant and by hypothesis, they are disjoint. By \cite{Kechris1995}*{Exercise~14.14}, there exists an invariant Borel set $C_r$ that separates them. Define $U \colon \cM(L_0) \to \R$ by
  \[
  U(A) = \sup \set{r : A \in C_r}.
  \]
  The function $U$ is $\cong$-invariant and Borel. So by Theorem~\ref{th:Lopez-Escobar}, there exists a sentence $\theta$ such that $U(A) = \theta^A$ and then for all separable $A$, $\phi_1^A \leq \theta^A \leq \phi_2^A$. Using the downwards Löwenheim--Skolem theorem, this implies that $\phi_1 \leq \theta \leq \phi_2$ is universally valid.
\end{proof}

Another corollary is that our universal modulus $\Omega$ is indeed universal.
\begin{cor}
  \label{c:universal-true}
  Let $L$ be a countable language and $\Omega$ be a universal modulus for $L$. Then every $\Lomo(\cL)$-sentence $\phi$ is equivalent to an $\Omega$-sentence $\phi'$. 
\end{cor}
\begin{proof}
  Apply Theorem~\ref{th:Lopez-Escobar} to the Borel function $U(A) = \phi^A$.
\end{proof}

%%%%%%%%%%%%%%%%%%%%%%%%%%%%%%%%%%%%%%%%%%%%%%%%%%

\section{Bounded rank and Borelness of isomorphism}
\label{sec:bounded-rank}

In this section, we characterize when the isomorphism relation restricted to an invariant Borel subset of $\cM$ is Borel and prove Theorem~\ref{th:intro:Borel-bounded}. This is again analogous to the classical setting (see, e.g., \cite{Gao2009a}*{Theorem~12.2.4}). Our  proof follows the same general outline, but certain new features appear.

For the remainder of the section, we fix a countable language $L$ and a weak modulus $\Omega$ universal for $L$. Recall that a \df{fragment} $F$ of $\Lomo(L)$ is a \emph{separable} set of formulas containing all atomic formulas and closed under subformulas, substitution of terms for variables, quantifiers, and finitary connectives. An \df{$F$-theory} is a collection of statements of the form $\phi = 0$, where $\phi$ is a sentence in $F$. If $A$ is a structure, the \df{theory of $A$}, denoted by $\Th(A)$, is the collection of all statements $\phi = 0$ that are true in $A$.

We will use a homogeneity result for atomic models. In order to state it correctly, we will need to define a topometric structure (see \cite{BenYaacov2008a}) on the (non-compact) type spaces, which we proceed to do. Let $T$ be an $F$-theory. Define a seminorm $\nm{\cdot}_T$ on the (unital) algebra of $n$-ary $F$-formulas by:
\begin{equation*}
  \nm{\phi}_T = \sup \set{|\phi^A(\bar a)| : A \models T, \bar a \in A^n}.
\end{equation*}
Denote by $\widehat{\tS_n(T)}$ the Gelfand space of the Hausdorff completion of the algebra of all $n$-ary $F$-formulas equipped with this seminorm. (See, for example, \cite{Folland1995}*{Chapter~1} for details on Gelfand theory.) This is the compact space of \df{finitely consistent $n$-types}. In what follows, we will identify the algebra of formulas with $C(\widehat{\tS_n(T)})$ via the Gelfand transform and simply use the notation $\phi(p)$ for the value that the type $p$ gives to the formula $\phi$. Alternatively, if $\phi(p) = r$, we will also write $p(\bar x) \models \phi(\bar x) = r$.

Say that a tuple $\bar a \in A^n$ \df{realizes} a type $p$ (notation $\bar a \models p$ or $\tp(\bar a) = p$) if $\phi(p) = \phi^A(\bar a)$ for all $n$-ary formulas $\phi \in F$. If there exists $\bar a \in A^n$ such that $\bar a \models p$, we will say that \df{$A$ realizes $p$}. Otherwise, say that $A$ \df{omits} $p$. Denote by $\tS_n(T)$ the set of \df{realizable} types (or just \df{types} for short):
\begin{equation*}
  \tS_n(T) = \set{p \in \widehat{\tS_n(T)} : \exists A \models T \text{ such that $A$ realizes } p}.
\end{equation*}
The \df{logic topology} on $\tS_n(T)$ is the one inherited from $\widehat{\tS_n(T)}$. The compactness theorem tells us that if $F = \Lfin(L)$, then $\tS_n(T) = \widehat{\tS_n(T)}$. While this fails for general fragments, $\tS_n(T)$ is always a $G_\delta$ subset of $\widehat{\tS_n(T)}$ and therefore a Polish space. (We will not prove this fact as we will not use it.)

Defining the distance on types is more delicate because of the lack of compactness. The definition we give is inspired by Caicedo and Iovino~\cite{Caicedo2014}. As in finitary continuous logic, the topology defined by the distance on types is finer than the logic topology and the distance is lower semicontinuous in the logic topology.

We will need to fix distances on powers of $A$ and for most purposes, any distance will do. However, in order to obtain exact equalities in the two propositions below, it will be most convenient to take $d = d^\Omega$; in the remainder of the section, when we write $d(\bar a, \bar b)$ for $\bar a, \bar b \in A^n$, we mean $d^\Omega(\bar a, \bar b)$ (and similarly, in formulas).

Recall that the operation $\dotminus$ is defined by $x \dotminus y = 0 \vee (x-y)$.
Let the distance $\partial_F$ on $\tS_n(T)$ be given by:
\begin{multline}
  \label{eq:type-distance}
  \partial_F(p, q) \leq s \iff \forall \phi \in F\ 
    q(\bar x) \models \qinf_{\bar y} (d(\bar x, \bar y) \dotminus s) \vee |\phi(\bar y) - \phi(p)| = 0.
\end{multline}
In words, $\partial_F(p, q) \leq s$ iff for every $\phi \in F$ and every $\eps > 0$, for every realization $\bar a \in A^n$ of $q$, there exists $\bar a' \in A^n$ such that $d(\bar a, \bar a') < s + \eps$ and $|\phi(\bar a') - \phi(p)| < \eps$. For $F = \Lfin$, using compactness, this definition is equivalent to the usual one (see \cite{BenYaacov2008}*{Section~8} for the definition). When the fragment $F$ is clear from the context, we will simply write $\partial$ instead of $\partial_F$.

We check that $\partial$ is symmetric. Suppose that $\partial(p, q) \leq s$ and fix $\phi \in F$, and a realization $\bar a \in A^n$ of $q$. Let $\psi(\bar x) = \inf_{\bar y}\, (d(\bar y, \bar x) \dotminus s) \vee |\phi(\bar y) - \phi(q)|$ and suppose, for contradiction, that $\psi(p) = r > 0$. Using \eqref{eq:type-distance} for $\psi$, we obtain that for every $\eps > 0$, there exists $\bar a' \in A^n$ with $d(\bar a, \bar a') < s + \eps$ and $\psi(\bar a') > r - \eps$. However, $\psi(\bar a') \leq d(\bar a, \bar a') \dotminus s$ (as $\phi(\bar a) = \phi(q)$), which yields a contradiction for $\eps < r/2$.

Next we verify the triangle inequality. Suppose that $\partial(p_1, p_2) \leq s_1$ and $\partial(p_2, p_3) \leq s_2$ in order to show that $\partial(p_1, p_3) \leq s_1 + s_2$. For simplicity of notation, we assume that $n=1$. Let $\phi \in F$ be a formula such that $\phi(p_1) = 0$. We know that
\begin{equation*}
  p_2(y) \models \qinf_w (d(y, w) \dotminus s_1) \vee \phi(w) = 0,
\end{equation*}
whence
\begin{equation*}
  p_3(x) \models \qinf_y (d(x, y) \dotminus s_2) \vee \big( \qinf_w (d(y, w) \dotminus s_1) \vee \phi(w) \big) = 0.
\end{equation*}
Simplifying yields
\begin{equation*}
  p_3(x) \models \qinf_w (d(x, w) \dotminus (s_1 + s_2)) \vee \phi(w) = 0,
\end{equation*}
as desired.

Another property of $\partial$ that we will need, easily checked from the definition, is that for any model $A \models T$ and all $\bar a, \bar b \in A^n$,
\begin{equation}
  \label{eq:type-dist-smaller-than-d}
  \partial(\tp \bar a, \tp \bar b) \leq d(\bar a, \bar b).
\end{equation}

Finally, we check that if $\phi$ is an $n$-ary $\Omega$-formula that is in $F$, $A, B \models T$, and $\bar a \in A^n, \bar b \in B^n$, then
\begin{equation}
  \label{eq:type-dist-dominates}
  |\phi(\bar a) - \phi(\bar b)| \leq \partial(\tp \bar a, \tp \bar b).
\end{equation}
Indeed, suppose that $\partial(\tp \bar a, \tp \bar b) \leq s$ and fix a formula $\phi$. Then for every $\eps > 0$, there exists $\bar b' \in B^n$ such that $d(\bar b', \bar b) < s + \eps$ and $|\phi(\bar b') - \phi(\bar a)| < \eps$. By the choice of the metric $d$ on products, $\phi$ is contractive in $d$, so we obtain that $|\phi(\bar a) - \phi(\bar b)| \leq s + 2 \eps$, which is enough.

If $X \sub \tS_n(T)$ and $\delta > 0$, denote by $(X)_\delta$ the \df{$\delta$-fattening} of $X$:
\begin{equation*}
  (X)_\delta = \set{p \in \tS_n(T) : \exists q \in X \ \partial(p, q) < \delta}.
\end{equation*}
If $X = \set{p}$ is a singleton, write $(p)_\delta$ instead of $(\set{p})_\delta$. Say that a type $p \in \tS_n(T)$ is \df{principal} if for every $\delta > 0$, $p$ belongs to the interior of $(p)_\delta$ in the logic topology on $\tS_n(T)$. A model $A$ is called \df{$F$-atomic} if every type it realizes is principal (for $\Th(A)$). One of the important properties of atomic models is that they are homogeneous.
\begin{prop}
  \label{p:atomic-homog}
  Let $F$ be a fragment, let $A$ be an $F$-atomic model, and $\bar u, \bar v \in A^k$. Then
  \begin{equation*}
    r_\infty(A \bar u, A \bar v) = \partial_F(\tp \bar u, \tp \bar v).
  \end{equation*}
  In particular, $A$ is \df{homogeneous}, i.e., for all $\bar u, \bar v \in A^k$ with $\tp \bar u = \tp \bar v$ and $\eps > 0$, there is an automorphism $f \in \Aut(A)$ such that $d(\bar u, f(\bar v)) < \eps$.
\end{prop}
\begin{proof}
  The inequality $\partial \leq r_\infty$ follows from Theorem~\ref{th:UniversalRInfty} and \eqref{eq:type-dist-smaller-than-d}. For the other direction, suppose that $\partial(\tp \bar u, \tp \bar v) < t$. We will build inductively tail-dense sequences $a, b \in A^\N$ that satisfy $a|_k = \bar u$, $b|_k = \bar v$, and
  \begin{equation}
    \label{eq:homog-bf}
    \partial \big( \tp(a|_n), \tp(b|_n) \big) < t \quad \text{ for all } n.
  \end{equation}
  We start by setting $a|_k = \bar u$ and $b|_k = \bar v$. The rest of the construction is carried by a back-and-forth argument of which we only describe the forth step. Suppose that $a|_n$, $b|_n$ are given such that $\partial(\tp a|_n, \tp b|_n) < s < t$ and let $a_n \in A$ be arbitrary. We will find $b_n \in A$ such that $\partial(\tp a|_{n+1}, \tp b|_{n+1}) < t$. As $p = \tp(a|_{n+1})$ is principal, there exists a formula $\phi$ taking non-negative values such that $\phi(a|_{n+1}) = 0$ and $\phi(\bar w) < 1$ implies that $\partial(\tp \bar w, p) < (t - s)/2$. As $\partial\big(\tp(a|_n), \tp(b|_n)\big) < s$, by \eqref{eq:type-distance}, we have that
  \begin{equation*}
    A \models \qinf_{\bar y} \big( (d(\bar y, b|_n) \dotminus s) \vee \inf_z \phi(\bar y, z) \big) = 0.
  \end{equation*}
  Let $\eps > 0$ be arbitrary and let $\bar c \in A^n$ and $d \in A$ be such that $d(b|_n, \bar c) < s + \eps$ and $\phi(\bar cd) < 1$. Then $\partial(\tp(\bar cd), p) < (t - s)/2$ and we have
  \begin{equation*}
    \begin{split}
      \partial(p, \tp(b|_n d)) &\leq \partial(p, \tp(\bar cd)) + \partial(\tp(\bar cd), \tp(\bar bd)) \\
      &\leq (t - s)/2 + s + \eps,
    \end{split}
  \end{equation*}
  which is less than $t$ as long as $\eps < (t-s)/2$. This means that we can take $b_n = d$.

  Now it only remains to observe that \eqref{eq:type-dist-dominates} implies that
  \begin{equation*}
    r_0(a|_n, b|_n) \leq \partial \big( \tp(a|_n), \tp(b|_n) \big)
  \end{equation*}
  and apply Proposition~\ref{p:BackAndForth}.
\end{proof}

The following proposition bounds the Scott rank of $F$-atomic models.
\begin{prop}
  \label{p:typed-rinfty}
  Let $F$ be a fragment such that the quantifier rank of formulas in $F$ is bounded by $\alpha$. Let $A$ be an $F$-atomic model. Then for every $\bar u, \bar v \in A^k$,
  \begin{equation*}
    \partial_F(\tp \bar u, \tp \bar v) = r_\alpha(A\bar u, A\bar v) = r_\infty(A\bar u, A\bar v).
  \end{equation*}
  In particular, the Scott rank of $A$ is at most $\alpha$.
\end{prop}
\begin{proof}
  In view of Proposition~\ref{p:atomic-homog}, we only need to prove that $\partial(\tp \bar u, \tp \bar v) \leq r_\alpha(\bar u, \bar v)$.
 Suppose that $r_\alpha(\bar u, \bar v) \leq s$ and let $\eps > 0$ be arbitrary. As $\tp \bar v$ is principal, there exists an $n$-ary $\Omega$-formula $\psi \in F$ and $\delta > 0$ such that $\psi(\bar v) = 0$ and $\psi(q) < \delta \implies \partial(q, \tp \bar v) < \eps$. Let $M > s/\delta$ and let
  \begin{equation*}
    \phi(\bar x) = \qinf_{\bar z} d(\bar x, \bar z) \vee M \cdot \psi(\bar z), 
  \end{equation*}
  where the variables $\bar z$ are taken in such a way that $\phi$ is an $n$-ary $\Omega$-formula (this can be done by Proposition~\ref{p:universal-mod}). Then $\phi(\bar v) = 0$ and as $\phi \in F$, by Theorem~\ref{th:r_alpha-qr}, we have that $\phi(\bar u) \leq s$. Then there exists $\bar w \in A^k$ such that $d(\bar u, \bar w) \leq s + \eps$ and $\psi(\bar w) < \delta$. Thus
  \begin{equation*}
    \begin{split}
    \partial(\tp \bar u, \tp \bar v) &\leq \partial(\tp \bar u, \tp \bar w) + \partial(\tp \bar w, \tp \bar v) \\
    &\leq d(\bar u, \bar w) + \partial(\tp \bar w, \tp \bar v) \\
    &\leq s + 2\eps.    
    \end{split}
  \end{equation*}
  As $\eps$ was arbitrary, this completes the proof.
\end{proof}

The following omitting types theorem will allow us to deduce that $F$-categorical models are $F$-atomic. The version for infinitary continuous logic was proved in Eagle~\cite{Eagle2014}*{Theorem~4.14}.
\begin{theorem}[Omitting types]
  \label{th:omitting-types}
  Let $T$ be a consistent $F$-theory and $p \in \tS_n(T)$ be a type which is not principal. Then there exists a separable model of $T$ which omits $p$.
\end{theorem}
Eagle's definition of a \df{metrically principal} type \cite{Eagle2014}*{Definition~4.12} which is used in his \cite{Eagle2014}*{Theorem~4.14} is not quite the same as our definition of a principal type. The following lemma shows that they are equivalent.
\begin{lemma}
  \label{l:principal-type}
  Let $p \in \tS_n(T)$. Then the following are equivalent:
  \begin{enumerate}
  \item \label{i:l:principal-type-1} $p$ is principal;
  \item \label{i:l:principal-type-2} for every $\delta > 0$, $(p)_\delta$ has non-empty interior in the logic topology.
  \end{enumerate}
\end{lemma}
\begin{proof}
  We only need to prove \ref{i:l:principal-type-2} $\Rightarrow$ \ref{i:l:principal-type-1}. Fix $1 > \delta > 0$ and suppose that $q_0$ is in the interior of $(p)_{\delta/2}$, i.e., there exists a formula $\phi(\bar x)$ such that $\phi(q_0) = 0$ and $\phi(q) < 1 \implies \partial(q, p) < \delta/2$. Let
  \begin{equation*}
    \psi(\bar x) = \qinf_{\bar y} \phi(\bar y) \vee (d(\bar x, \bar y) \dotminus \delta/2).
  \end{equation*}
  First we check that $\psi(p) = 0$. Let $\bar a \models p$. Then, as $\partial(p, q_0) < \delta/2$ and $\phi(q_0) = 0$, we have that $\psi(p) = 0$. Next we verify that $\psi(q) < \delta/2 \implies q \in (p)_\delta$, thus showing that $p$ is in the interior of $(p)_\delta$. Suppose that $\psi(q) < \delta/2$ and let $\bar a \models q$. Then there exists $\bar b$ such that $\phi(\bar b) < \delta/2 < 1$ and $d(\bar a, \bar b) < \delta/2$.  In particular, $\partial(p, tp b) < \delta/2$. We have
  \begin{equation*}
    \partial(p, q) \leq \partial(p, \tp \bar b) + \partial(\tp \bar b, \tp \bar a) < \delta/2 + \delta/2 = \delta. \qedhere
  \end{equation*}
\end{proof}

Combining everything we have so far, we obtain the following theorem.
\begin{theorem}
  \label{th:Scott-rank-categorical}
   Let $F$ be a fragment such that the quantifier rank of formulas in $F$ is bounded by $\alpha$ and let $T$ be an $F$-theory which has a unique separable model $A$. Then the Scott rank of $A$ is at most $\alpha$.
 \end{theorem}
 \begin{proof}
   By Theorem~\ref{th:omitting-types}, $A$ is $F$-atomic: if $\tp \bar a$ is non-principal for some $\bar a \in A^n$, there exists a separable model of $T$ that omits it and is therefore not isomorphic to $A$. Now Proposition~\ref{p:typed-rinfty} implies the conclusion.
 \end{proof}

We are finally ready to prove the main theorem of this section.
\begin{theorem}
  \label{th:Borel-isom}
  Let $L$ be a countable language, let $\cong$ denote the isomorphism relation on $\cM(L)$ and let $X \sub \cM(L)$ be an $\cong$-invariant Borel subset. Let $\Omega$ be a universal modulus for $L$. Then the following are equivalent:
  \begin{enumerate}
  \item \label{i:th:Borel-isom-1} $\cong|_X$ is Borel;
  \item \label{i:th:Borel-isom-2} The $\Omega$-Scott rank of elements of $X$ is uniformly bounded below $\omega_1$.
  \end{enumerate}
\end{theorem}
\begin{proof}
  \begin{cycprf}
  \item[\impnext] Suppose that $\cong|_X$ is $\bPi^0_\alpha$ for some $\alpha < \omega_1$. Then each isomorphism class contained in $X$ is $\bPi^0_\alpha$; by Theorem~\ref{th:Lopez-Escobar}, for every $A \in X$, there exists a sentence $\psi_A$ of quantifier rank at most $\alpha' = \omega \cdot 2 \cdot (2 + \alpha + 1)$ such that $\psi_A^B = 0$ if $A \cong B$ and $\psi_A^B = 1$ otherwise. Now for each $A \in X$, apply Theorem~\ref{th:Scott-rank-categorical} to the fragment generated by $\psi_A$ and the theory $\set{\psi_A = 0}$ to obtain that the rank of $A$ is at most $\alpha'$.

  \item[\impfirst] Suppose that the Scott rank of all structures in $X$ is bounded by $\alpha < \omega_1$. Then the Scott sentence of each of those structures has quantifier rank at most $\alpha + \omega$. By Theorems \ref{th:Scott-sentence} and \ref{th:r_alpha-qr}, for $A, B \in X$, $A \eqrel{E_{\alpha + \omega}} B$ iff $A \eqrel{E_\infty} B$ and by Theorem~\ref{th:UniversalRInfty}, $A \eqrel{E_\infty} B$ iff $A \cong B$. Thus, on $X$, the Borel relation $E_{\alpha + \omega}$ coincides with isomorphism.
  \end{cycprf}
\end{proof}

\begin{example} \label{ex:proper metric} Recall that a complete metric space is called \df{proper} (or Heine--Borel) if all closed bounded sets are compact. The Euclidean spaces and, more generally, all complete Riemannian manifolds  are examples of proper metric spaces.   Clearly it suffices to require the Heine--Borel condition for closed balls centered at points in a countable dense sequence.
Then, since a complete metric space is compact iff it is totally bounded, properness is a  Borel property   of Polish metric spaces.

By a theorem of Hjorth (see \cite[Thm.\ 3]{Gao.Kechris:03}), isometry of proper metric spaces is Borel bireducible with the universal countable Borel equivalence relation, and in particular Borel. Hence by Theorem~\ref{th:Borel-isom}, the Scott rank of proper metric spaces is bounded.
%A similar argument works for locally compact connected Polish metric spaces.
\end{example}

%%%%%%%%%%%%%%%%%%%%%%%%%%%%%%%%%%%%%%%%%%%%%%%%%%

\section{Two examples for the $1$-Lipschitz modulus}
\label{sec:some-examples}

\subsection{Gromov--Hausdorff distance between metric spaces}
\label{sec:grom-hausd-dist}

We apply the tools we have developed in a particular example: calculating the Gromov--Hausdorff distance between metric spaces. Throughout this subsection, we fix a signature $L$ containing only the distance symbol $d$, say, bounded by $1$, and we let $\Omega = \OmegaL$ be the $1$-Lipschitz weak modulus defined by \eqref{eq:OmegaL}. All metric spaces that we consider below have diameter bounded by $1$. (Everything goes through if $1$ is replaced by an arbitrary positive constant.)

Recall that if $(C, d)$ is a metric space and $A, B$ are closed subsets of $C$, the \df{Hausdorff distance} between $A$ and $B$, denoted by $\dH(A, B)$ is defined by
\begin{equation*}
  \dH(A, B) = \sup_{x \in A} d(x, B) \vee \sup_{y \in B} d(y, A).
\end{equation*}
If $(A, d)$ and $(B, d)$ are now abstract metric spaces, the \df{Gromov--Hausdorff} distance  \cite[Def.\ 3.4]{Gromov:07} between $A$ and $B$, denoted by $\dGH(A, B)$, is defined by
\begin{equation*}
  \dGH(A, B) = \inf_{f_1, f_2} \dH(f_1(A), f_2(B)),
\end{equation*}
where $f_1$ and $f_2$ vary over all isometric embeddings of $A$ and $B$ in a third space $C$.

Similarly, we define the \df{enumerated Gromov--Hausdorff distance}  between enumerated metric spaces as follows. Let $A$ and $B$ be metric spaces and $\set{a_i : i \in I}$, $B = \set{b_i : i \in I}$ be sequences of elements of $A$ and $B$, respectively. We define the \emph{enumerated Gromov--Hausdorff distance} $\deGH(a, b)$ as
\begin{equation*}
  \deGH(a, b) = \inf_{f_1,f_2} \, \sup_i d\bigl( f_1(a_i), f_2(b_i) \bigr),
\end{equation*}
where $f_1$ and $f_2$ vary over all isometric embeddings of $\set{a_i : i \in I}$ and $\set{b_i : i \in I}$ in a third metric space $C$.

\begin{theorem}
  \label{thm:ScottLipschitzGromovHausdorff}
  Let $L = \set{d}$ and $\Omega = \OmegaL$ as above. Then for any two metric spaces $A$, $B$ of diameter at most $1$,
  \begin{equation*}
    r_\infty(A, B) = \dGH(A, B).
  \end{equation*}
\end{theorem}

We start by calculating $\deGH$ between finite tuples.
\begin{lemma}
  \label{l:enum-GH}
  Let $A, B$ be metric spaces and let $\bar a \in A^n$, $\bar b \in B^n$ be finite tuples. Then
  \begin{equation*}
    \deGH(\bar a, \bar b) = \half \sup_{i,j} \, \bigl| d(a_i,a_j) - d(b_i,b_j) \bigr| = r_0(A \bar a, B \bar b).
  \end{equation*}
\end{lemma}
\begin{proof}
  The first equality is proved in Uspenskij~\cite{Uspenskij2008}*{Proposition~7.1}; we proceed to show the second.
  As basic formulas are $1$-Lipschitz and they are preserved by embeddings, it is clear  that $r_0(A \bar a, B \bar b) \leq \deGH(\bar a, \bar b)$.
  
  On the other hand, we have
  \begin{equation*}
    \Delta_{\half d(x_i,x_j)}(\delta) = \half (\delta_i+\delta_j) \leq \delta_i \vee \delta_j \leq \OmegaL(\delta),
  \end{equation*}
  so $\half d(x_i,x_j)$ is an $\OmegaL$-formula. This implies the other inequality.
\end{proof}

\begin{proof}[Proof of Theorem~\ref{thm:ScottLipschitzGromovHausdorff}]
  Let $A, B$ be metric spaces and $s \in \R$. We have:
  \begin{equation*}
    \begin{split}
      \dGH(A, B) < s &\iff \exists a \in A^\N, b \in B^\N \text{ tail-dense such that } \deGH(a, b) < s \\
      &\iff \exists a \in A^\N, b \in B^\N \ \sup_n \deGH(a|_n, b|_n) < s \\
      &\iff \exists a \in A^\N, b \in B^\N \ \sup_n r_0(a|_n, b|_n) < s \\
      &\iff r_\infty(A, B) < s.
    \end{split}
  \end{equation*}
  The second equivalence follows by compactness (or, alternatively, from the fact that the first equality in Lemma~\ref{l:enum-GH} also holds for infinite tuples), the third is Lemma~\ref{l:enum-GH}, and the fourth is given by Proposition~\ref{p:BackAndForth}.
\end{proof}

Theorem~\ref{thm:ScottLipschitzGromovHausdorff} and Corollary~\ref{c:classes-Borel} give us the following.
\begin{cor}
  \label{c:dGH-Borel}
  Let $L = \set{d}$ and let $A \in \cM(L)$. Then the set
  \begin{equation*}
    \set{B \in \cM(L) : \dGH(A, B) = 0}
  \end{equation*}
  is Borel.
\end{cor}

\begin{question}
  \label{q:dGH-Borel}
  Let $\cM$ be the space of Polish metric spaces with distance bounded by $1$, as above. Let $A \in \cM$ be fixed. Is $\dGH(A, \cdot)$ a Borel function on $\cM$?
\end{question}

\begin{remark}
  \label{rem:EGH}
  It is not clear what the exact complexity of the equivalence relation $E_{\mathrm{GH}}$ ($\dGH = 0$ on bounded metric spaces) is. However, Christian Rosendal pointed out to us that it is above the universal orbit equivalence relation of a Polish group action. This can be seen as follows. For a compact metrizable space $X$, denote by $C(X)$ the space of continuous functions on $X$ equipped with the $\sup$ norm. Dutrieux and Kalton show in \cite{Dutrieux2005} that if $X$ and $Y$ are not homeomorphic, then $\dGH(C(X), C(Y)) \geq 1/16$. Thus the map $X \mapsto C(X)$ is a reduction from homeomorphism of compact spaces to $E_{\mathrm{GH}}$. As the former is universal for orbit equivalence relations of Polish group actions (Zielinski~\cite{Zielinski2016}), this yields the claim.

  One way to see that this reduction is Borel is as follows. In \cite{Zielinski2016}, compact metrizable spaces are parametrized by the elements of the hyperspace $K([0, 1]^\N)$ of closed subspaces of the Hilbert cube. An easy way to find a Borel map $\Phi \colon K([0, 1]^\N) \to \cM$ such that $\Phi(X)$ is isometric to $C(X)$ is the following. Let $z_i \colon [0, 1]^\N \to [0, 1]$ denote the projection on the $i$-th coordinate. Then instead of $\N$, we can take as a distinguished countable dense set in the definition of $\cM$ the set of all polynomials with rational coefficients in infinitely many variables (all but finitely many coefficients of which are zero), and calculate the distance by
  \begin{equation*}
    d(P_1, P_2) = \sup_{x \in X} |P_1(z_0(x), z_1(x), \ldots) - P_2(z_0(x), z_1(x), \ldots)|,
  \end{equation*}
  which is easily a Borel function of $X$. The Stone--Weierstrass theorem tells us that the resulting element of $\cM$ is isometric to $C(X)$.
\end{remark}

\begin{question}
  \label{q:complexity-EGH}
  Is $E_{\mathrm{GH}}$ Borel reducible to an orbit equivalence relation of a Polish group action?
\end{question}

Next we provide a simple example of two discrete metric spaces that have Gromov--Hausdorff distance $0$ but are not isometric. Note however, that for compact metric spaces as well as for metric spaces with discrete set of distances, $E_{\mathrm{GH}}$ and isometry coincide.
\begin{example}
  Let $A$ be a countable subset of the open interval $(0,1)$. We define a countable metric space $(X_A,d_A)$ as follows: $X_A=\{x_a,y_a:a\in A\}$, and for $x\neq y\in X_A$ we set
  \begin{equation*}
    d_A(x,y)=\begin{cases} a &\text{if }x=x_a,y=y_a \text{ or vice versa}; \\
      1& \text{otherwise.}
    \end{cases}    
  \end{equation*}
  It is easy to check that $(X_A, d_A)$ is a complete, ultrametric space. Now let $A, B \subseteq (0,1)$ be two distinct, dense, countable subsets of $(0,1)$. Then clearly $X_A$ and $X_B$ are not isometric but $\dGH(X_A, X_B) = 0$. Indeed, fix $\eps > 0$ and let  $f \colon A \to B$ be a bijection such that for all $a \in A$, $|f(a) - a| < \eps$. Let $Y$ be the disjoint union of $X_A$ and $X_B$ and define the distance on $Y$ by $d(x_a, y_a) = a$, $d(x_b, y_b) = b$, for $a \in A$, $b \in B$; $d(x_a, x_{f(a)}) = d(y_a, y_{f(a)}) = \eps$, $d(x_a, y_{f(a)}) = d(y_a, x_{f(a)}) = \min ((a+b)/2+\eps, 1)$ for all $a \in A$ and $d(x, y) = 1$ in all other cases. This $Y$ witnesses that $\dGH(X_A, X_B) \leq \eps$.
\end{example}

%%%%%%%%%%%%%%%%%%%%%%%%%%%%%%%%%%%%%%%%%%%%%%%%%%

\subsection{Kadets distance between Banach spaces}
\label{sec:banach-spaces}

Our second example is Banach spaces. Here the language is $\set{0, +, \nm{\cdot}_n} \cup \set{M_\lambda : \lambda \in \R}$, where $0$ is a constant symbol, $+$ is a binary operation with $\Delta_+(\delta_1, \delta_2) = \delta_1 + \delta_2$, $\nm{\cdot}_n$ is a $1$-Lipschitz predicate symbol (interpreted as $\nm{x}_n = \nm{x} \wedge n$), and $M_\lambda$ is a function symbol representing multiplication by $\lambda$ with $\Delta_{M_\lambda}(\delta_0) = |\lambda| \delta_0$. We will denote by $\cM$ (and restrict our considerations to) the Borel set of all separable Banach spaces rather than the space of all structures in this signature. Note that this is not relevant for the definition of the pseudo-distances $r_\alpha$: they are computed in a manner independent of the ambient space. We can also restrict to a countable sublanguage by only keeping $M_\lambda$ for $\lambda \in \Q$ without loss of generality. The distance function is given by $d(x, y) = \nm{x - y}_1$.

Every atomic formula is equivalent (on $\cM$) to a formula of the type $\nm{\sum_{i = 0}^{n-1} \lambda_i x_i}_k$ for some $k, n \in \N$, $\lambda_i \in \R$. An easy calculation yields that the modulus of continuity of this formula is $\Delta(\delta) = \sum_{i = 0}^{n-1} |\lambda_i|\delta_i$. In particular, it obeys $\OmegaL$ iff $\sum_i |\lambda_i| \leq 1$.

Let $A, B$ be two Banach spaces. The Kadets distance between $A$ and $B$ is defined analogously to the Gromov--Hausdorff distance by
\begin{equation*}
\dK(A, B) = \inf_{f_1, f_2} \dH(f_1(\cB_A), f_2(\cB_B)),
\end{equation*}
where $\cB_A$ and $\cB_B$ denote the unit balls of $A$ and $B$ and $f_1$ and $f_2$ vary over all \emph{linear} isometric embeddings of $A$ and $B$ in a third Banach space $C$.

Let now $A$ and $B$ be two Banach spaces and $\bar a$ and $\bar b$ be two sequences in $\cB_A$ and $\cB_B$ with dense span. The \emph{enumerated Kadets distance} is defined as
\begin{equation*}
\deK(A \bar a, B \bar b) = \inf_{f_1, f_2}\sup_i \nm{f_1(a_i) - f_2(b_i)}_C,
\end{equation*}
where $f_1$ and $f_2$ vary over all linear isometric embeddings of $A$ and $B$ in a third space $C$. The following is the analogue of Lemma~\ref{l:enum-GH} for Banach spaces. 
\begin{lemma}
  \label{l:enum-Kadets}
  Let $A$ and $B$ be Banach spaces and $\bar a \in \cB_A^n$, $\bar b \in \cB_B^n$ be finite tuples. Then
  \begin{equation*}
    \begin{split}
      \deK(A\bar a, B \bar b) &= \sup\{\big| \nm{\sum_i \lambda_i a_i} - \nm{\sum_i \lambda_i b_i} \big| :  \sum_i |\lambda_i| \leq 1\} \\
      &= r_0(A \bar a, B \bar b).      
    \end{split}
  \end{equation*}
\end{lemma}
\begin{proof}
  The first equality is proved in \cite{BenYaacov2014}*{Fact~3.4}. (There, the $\sup$ is taken over all $\lambda_i$ with $\sum \lambda_i = 1$; however, scaling up in order to make the sum equal to $1$ only increases the values inside the $\sup$.) For the second, exactly as in Lemma~\ref{l:enum-GH}, it suffices to recall that the formula $\nm{\sum_i \lambda_i x_i}$ obeys $\OmegaL$ if $\sum_i{|\lambda_i|} \leq 1$.
\end{proof}

In the same way as before, we obtain the following theorem.
\begin{theorem}
  \label{th:Kadets}
  For any two separable Banach spaces $A$ and $B$, we have
  \begin{equation*}
    r^{\OmegaL}_\infty(A, B) = \dK(A, B).
  \end{equation*}
\end{theorem}

Question~\ref{q:dGH-Borel}, Remark~\ref{rem:EGH}, and Question~\ref{q:complexity-EGH} also apply to the Kadets distance.

%%%%%%%%%%%%%%%%%%%%%%%%%%%%%%%%%%%%%%%%%%%%%%%%%%

\section{A characterization of CLI Polish groups}
\label{sec:char-cli-polish}

Recall that a Polish group is called a \df{CLI group} if it admits a complete, compatible, left-invariant metric. In \cite{Gao1998}, Gao proves that if $M$ is a classical countable structure, the automorphism group $\Aut(M)$ is a CLI group iff the classical Scott sentence of $M$ does not admit uncountable models. In this section, we note that a similar result holds for general separable structures.

We start with a simple lemma (well-known in the classical setting) that establishes the connection between the left completion of $\Aut(M)$ and $\Lomo$-logic.
\begin{lemma}
  \label{l:left-completion}
  Let $M$ be a separable metric structure and $\Aut(M)$ be its automorphism group equipped with the pointwise convergence topology. Then the left completion of $\Aut(M)$ can be identified with the monoid of $\Lomo$-elementary embeddings of $M$ into itself. In particular, $\Aut(M)$ is a CLI group iff every such elementary embedding is surjective.
\end{lemma}
\begin{proof}
  The left uniformity of $\Aut(M)$ is defined by the pseudo-distances $\rho_a$ given by $\rho_a(g, h) = d(g \cdot a, h \cdot a)$ for $a \in M$. If $(g_n)_n$ is a left-Cauchy sequence in $\Aut(M)$, then $g_n \cdot a$ converges for every $a \in M$; denote by $f$ the limit map. As the interpretation of any $\Lomo$-formula $\phi(\bar x)$ is continuous, we have that
  \begin{equation*}
    \phi^M(f(\bar a)) = \lim_n \phi^M(g_n \cdot \bar a) = \phi^M(\bar a)
  \end{equation*}
  for any tuple $\bar a \in M^k$. We conclude that $f$ is elementary.

  Conversely, let $f \colon M \to M$ be $\Lomo$-elementary; we will show that it can be arbitrarily well approximated by automorphisms. Let $\bar a \in M^k$ be any tuple. By elementarity and Theorem~\ref{th:Scott-sentence}, we have that $r^{\OmegaU(L)}_\infty(M\bar a, Mf(\bar a)) = 0$; now Theorem~\ref{th:UniversalRInfty} implies that for any $\eps > 0$, there exists an automorphism $g$ of $M$ such that $d(f(\bar a), g \cdot \bar a) < \eps$.

  For the second claim of the lemma, note that $\Aut(M)$ is CLI iff it coincides with its left completion.
\end{proof}

\begin{theorem}
  \label{th:Gao}
  Let $M$ be a separable structure in a countable language. Then the following are equivalent:
  \begin{enumerate}
  \item The Scott sentence of $M$ is non-zero on non-separable models;
  \item \label{i:Gao:CLI} $\Aut(M)$ is a CLI group.
  \end{enumerate}
\end{theorem}
\begin{proof}
  The proof is similar to the one in \cite{Gao1998}, so we only give a brief sketch. By Lemma~\ref{l:left-completion}, we can replace \ref{i:Gao:CLI} by the condition that every $\Lomo$-elementary embedding of $M$ into itself is surjective.
  
  \begin{cycprf}
  \item[\hspace{0.6\itemindent} \impnext] Suppose that $f \colon M \to M$ is an elementary embedding which is not surjective. By iterating $f$ $\omega_1$ times, we obtain an elementary chain every element of which is isomorphic to $M$; its union is non-separable and elementarily equivalent to $M$.
    
  \item[\impfirst] Conversely, suppose that $N$ is a non-separable model elementarily equivalent to $M$. Let $F$ denote the fragment of $\Lomo$ generated by the Scott sentence of $M$. Then by the downward Löwenheim--Skolem theorem, $M$ can be embedded $F$-elementarily into $N$. Let $a \in N \sminus M$ be arbitrary. Again by Löwenheim--Skolem, there exists a separable $M' \prec_F N$ such that $M \cup \set{a} \sub M'$. Then $M' \cong M$ (as $M'$ satisfies the Scott sentence of $M$) and the embedding $M \sub M'$ is not surjective.
  \end{cycprf}
\end{proof}

 As isometry groups of proper metric spaces (see Example~\ref{ex:proper metric} for the definition) are locally compact, and locally compact groups are CLI, we have the following corollary.
\begin{cor}
  \label{c:proper-metrics}
  Let $M$ be a separable structure such that the underlying metric space is proper. Then $M$ is the unique model of its  Scott sentence.
\end{cor}
As noted in \cite{Gao1998}, Hjorth and Solecki proved that all solvable groups are CLI. This  provides a further setting where Theorem~\ref{th:Gao} applies.

%%%%%%%%%%%%%%%%%%%%%%%%%%%%%%%%%%%%%%%%%%%%%%%%%%

\section{Connections with classical logic}
\label{sec:classical-logic}

One might ask about the connection between the continuous $\Lomo$-logic considered in this paper and classical ($\set{0, 1}$-valued) $\Lomo$-logic with regard to Polish metric structures. The latter has already been considered in the literature (see \cite{Doucha2014} and the references therein). The main difference is that, while continuous logic treats separable structures as essentially countable objects, in classical logic, which does not allow for approximations, they are basically discrete structures of size continuum and as a result, quite unmanageable from a descriptive set theoretic point of view. Classically $\Lomo$-definable sets of Polish structures are usually not Borel sets. For a natural example, take local compactness of Polish metric spaces: by \cite{Nies.Solecki:15}, the class of such spaces is properly co-analytic, and at the same time definable in classical $\Lomo$.

In this section, we observe that, with appropriate coding, the expressive power of classical $\Lomo$-logic is strictly greater than that of continuous $\Lomo$-logic. Each continuous signature $L$ yields a classical signature $\hat L$ defined as follows. For every predicate symbol $P$ of arity $n_P$ and for every $q \in \Q$, we put into $L'$ an $n_P$-ary relation symbol $P_q$. (For simplicity, we assume that the language $L$ has no function symbols; as we saw in Section~\ref{sec:space-polish-struct}, this entails no loss of generality.) To every $L$-structure $A$, we associate a classical $\hat L$-structure $\hat A$ with the same domain, where the interpretations of the symbols are given by:
\begin{equation*}
\hat  A \models P_q(\bar a) \iff P^A(\bar a) < q
\end{equation*}
for each $q\in \Q$ and $\bar a \in A^n$. 

\begin{prop}
  \label{p:classical-Lomo}
  For every $n$-ary $\Lomo(L)$-formula $\phi(\bar x)$ and every $q \in \Q$, there exists a formula $\hat \phi_q(\bar x)$ in  $\Lomo(\hat L)$ such that for every $L$-structure $A$, $\bar a \in A^n$ and each $q$, we have
  \begin{equation*}
    \hat A \models \hat \phi_q(\bar a) \iff \phi^A(\bar a) < q.
  \end{equation*}
  Moreover, the quantifier rank of $\hat \phi_q$ is equal to that of $\phi$.
\end{prop}
\begin{proof}
  First note that if we have formulas $\hat \phi_q$ for every $q$, then we can easily find a formula $\hat \phi^q$ that expresses $\phi > q$, namely $\hat \phi^q = \bigvee_{p > q} \neg \hat \phi_p$.
  
  The proof of the proposition proceeds by induction on $\phi$. For atomic formulas this is true by definition. If $\phi(\bar x) = \qinf_y \psi(\bar x, y)$, we can take $\hat \phi_q(\bar x) = \exists y \ \psi_q(\bar x, y)$. If $\phi = \bigwedge_i \psi_i$, we can take $\phi_q = \bigvee_i \psi_{i, q}$.

  Finally, suppose that $\phi = f(\psi_0, \ldots, \psi_{k-1})$, where $f \colon \R^k \to \R$ is a connective. Write
  \begin{equation*}
    f^{-1}\big((-\infty, q)\big) = \bigcup_{i \in \N} \prod_{j < k} (p_{i, j}, q_{i, j}),
  \end{equation*}
  where $p_{i, j}, q_{i, j} \in \Q$. This can be done because $f^{-1}\big((-\infty, q)\big)$ is an open set and products of rational intervals form a basis for $\R^k$. Then we can take
  \begin{equation*}
    \hat \phi_q = \bigvee_i \bigwedge_{j < k} \hat \psi_j^{p_{i, j}} \land \hat \psi_{j, q_{i, j}}. \qedhere
  \end{equation*}
\end{proof}

Thus, using Theorem~\ref{th:Lopez-Escobar}, one obtains for every invariant Borel set of models a classical sentence that describes it. However, there is no converse to this: the set of models of a classical $\Lomo$-sentence is in general not Borel, as mentioned above. From Corollary~\ref{c:iso-classes-Borel}, we obtain the following.
\begin{cor}
  \label{c:classical-Scott-sentence}
  For every separable structure $A$, there exists a sentence $\sigma_A$ in classical $\Lomo$-logic such that for every separable structure B,
  \begin{equation*}
    \hat B \models \sigma_A \iff B \cong A.
  \end{equation*}
\end{cor}

Finally, note that the classical Scott rank (as discussed in \cite{Doucha2014}) and the continuous Scott rank from this paper are different.
For example, consider the structure  $G_L$ from Remark~\ref{rem:non-closed-orbits}: as it is ultrahomogeneous, its continuous Scott rank is $0$. However, as orbits of the automorphism group are not closed, its classical Scott rank is non-zero. It was proved in \cite{Doucha2014} that the classical Scott rank of Polish metric spaces is at most $\omega_1$ but it remains an open question whether it must be countable.

%\bibliography{refs}

\begin{bibdiv}
\begin{biblist}

\bib{BenYaacov2008}{incollection}{
      author={{Ben Yaacov}, Ita{\"{\i}}},
      author={Berenstein, Alexander},
      author={Henson, C.~Ward},
      author={Usvyatsov, Alexander},
       title={Model theory for metric structures},
        date={2008},
   booktitle={Model theory with applications to algebra and analysis. {V}ol.
  2},
      series={London Math. Soc. Lecture Note Ser.},
      volume={350},
   publisher={Cambridge Univ. Press},
     address={Cambridge},
       pages={315\ndash 427},
         url={http://dx.doi.org/10.1017/CBO9780511735219.011},
}

\bib{BenYaacov2009a}{article}{
      author={{Ben Yaacov}, Ita{\"{\i}}},
      author={Iovino, Jos{\'e}},
       title={Model theoretic forcing in analysis},
        date={2009},
        ISSN={0168-0072},
     journal={Ann. Pure Appl. Logic},
      volume={158},
      number={3},
       pages={163\ndash 174},
         url={http://dx.doi.org/10.1016/j.apal.2007.10.011},
}

\bib{BenYaacov2008a}{article}{
      author={Ben~Yaacov, Ita{\"{\i}}},
       title={Topometric spaces and perturbations of metric structures},
        date={2008},
        ISSN={1863-3617},
     journal={Log. Anal.},
      volume={1},
      number={3-4},
       pages={235\ndash 272},
         url={http://dx.doi.org/10.1007/s11813-008-0009-x},
}

\bib{BenYaacov2013b}{article}{
      author={Ben~Yaacov, Ita{\"{\i}}},
       title={Lipschitz functions on topometric spaces},
        date={2013},
        ISSN={1759-9008},
     journal={J. Log. Anal.},
      volume={5},
       pages={Paper 8, 21},
}

\bib{BenYaacov2014}{article}{
      author={Ben~Yaacov, Ita{\"{\i}}},
       title={The linear isometry group of the {G}urarij space is universal},
        date={2014},
        ISSN={0002-9939},
     journal={Proc. Amer. Math. Soc.},
      volume={142},
      number={7},
       pages={2459\ndash 2467},
         url={http://dx.doi.org/10.1090/S0002-9939-2014-11956-3},
}

\bib{Caicedo2014}{article}{
      author={Caicedo, Xavier},
      author={Iovino, Jos{\'e}~N.},
       title={Omitting uncountable types and the strength of {$[0,1]$}-valued
  logics},
        date={2014},
        ISSN={0168-0072},
     journal={Ann. Pure Appl. Logic},
      volume={165},
      number={6},
       pages={1169\ndash 1200},
         url={http://dx.doi.org/10.1016/j.apal.2014.01.005},
      review={\MR{3183322}},
}

\bib{Coskey2016}{article}{
      author={Coskey, Samuel},
      author={Lupini, Martino},
       title={A {L}\'opez-{E}scobar theorem for metric structures, and the
  topological {V}aught conjecture},
        date={2016},
        ISSN={0016-2736},
     journal={Fund. Math.},
      volume={234},
      number={1},
       pages={55\ndash 72},
         url={http://dx.doi.org/10.4064/fm135-1-2016},
      review={\MR{3509816}},
}

\bib{Doucha2014}{article}{
      author={Doucha, Michal},
       title={Scott rank of {P}olish metric spaces},
        date={2014},
        ISSN={0168-0072},
     journal={Ann. Pure Appl. Logic},
      volume={165},
      number={12},
       pages={1919\ndash 1929},
         url={http://dx.doi.org/10.1016/j.apal.2014.08.002},
        note={Erratum \cite{Doucha2017}},
}

\bib{Doucha2017}{article}{
      author={Doucha, Michal},
       title={Erratum to: ``{S}cott rank of {P}olish metric spaces'' [{A}nn.
  {P}ure {A}ppl. {L}ogic 165 (12) (2014) 1919--1929] [ {MR}3256743]},
        date={2017},
        ISSN={0168-0072},
     journal={Ann. Pure Appl. Logic},
      volume={168},
      number={7},
       pages={1490},
         url={http://dx.doi.org/10.1016/j.apal.2017.03.006},
      review={\MR{3638899}},
}

\bib{Dutrieux2005}{article}{
      author={Dutrieux, Yves},
      author={Kalton, Nigel~J.},
       title={Perturbations of isometries between {$C(K)$}-spaces},
        date={2005},
        ISSN={0039-3223},
     journal={Studia Math.},
      volume={166},
      number={2},
       pages={181\ndash 197},
         url={http://dx.doi.org/10.4064/sm166-2-4},
      review={\MR{2109588}},
}

\bib{Eagle2014}{article}{
      author={Eagle, Christopher~J.},
       title={Omitting types for infinitary {$[0,1]$}-valued logic},
        date={2014},
        ISSN={0168-0072},
     journal={Ann. Pure Appl. Logic},
      volume={165},
      number={3},
       pages={913\ndash 932},
         url={http://dx.doi.org/10.1016/j.apal.2013.11.006},
      review={\MR{3142393}},
}

\bib{Elliott2013}{article}{
      author={Elliott, G.~A.},
      author={Farah, I.},
      author={Paulsen, V.~I.},
      author={Rosendal, C.},
      author={Toms, A.~S.},
      author={Törnquist, A.},
       title={The isomorphism relation for separable {C}*-algebras},
        date={2013},
     journal={Math. Res. Lett.},
      volume={20},
      number={6},
       pages={1071\ndash 1080},
}

\bib{Folland1995}{book}{
      author={Folland, Gerald~B.},
       title={A course in abstract harmonic analysis},
      series={Studies in Advanced Mathematics},
   publisher={CRC Press},
     address={Boca Raton, FL},
        date={1995},
        ISBN={0-8493-8490-7},
}

\bib{Farah2014}{article}{
      author={Farah, Ilijas},
      author={Hart, Bradd},
      author={Sherman, David},
       title={Model theory of operator algebras {II}: model theory},
        date={2014},
        ISSN={0021-2172},
     journal={Israel J. Math.},
      volume={201},
      number={1},
       pages={477\ndash 505},
         url={http://dx.doi.org/10.1007/s11856-014-1046-7},
}

\bib{Gao1998}{article}{
      author={Gao, Su},
       title={On automorphism groups of countable structures},
        date={1998},
        ISSN={0022-4812},
     journal={J. Symbolic Logic},
      volume={63},
      number={3},
       pages={891\ndash 896},
}

\bib{Gao2009a}{book}{
      author={Gao, Su},
       title={Invariant descriptive set theory},
      series={Pure and Applied Mathematics (Boca Raton)},
   publisher={CRC Press, Boca Raton, FL},
        date={2009},
      volume={293},
        ISBN={978-1-58488-793-5},
}

\bib{Gromov:07}{book}{
      author={Gromov, Misha},
       title={Metric structures for {R}iemannian and non-{R}iemannian spaces},
     edition={English},
      series={Modern Birkh\"auser Classics},
   publisher={Birkh\"auser Boston Inc.},
     address={Boston, MA},
        date={2007},
        ISBN={978-0-8176-4582-3; 0-8176-4582-9},
        note={Based on the 1981 French original, With appendices by M. Katz, P.
  Pansu and S. Semmes, Translated from the French by Sean Michael Bates},
      review={\MR{2307192 (2007k:53049)}},
}

\bib{Gao.Kechris:03}{article}{
      author={Gao, Su},
      author={Kechris, Alexander~S.},
       title={On the classification of {P}olish metric spaces up to isometry},
        date={2003},
        ISSN={0065-9266},
     journal={Mem. Amer. Math. Soc.},
      volume={161},
      number={766},
       pages={viii+78},
      review={\MR{1950332 (2004b:03067)}},
}

\bib{Hjorth2000}{book}{
      author={Hjorth, Greg},
       title={Classification and orbit equivalence relations},
      series={Mathematical Surveys and Monographs},
   publisher={American Mathematical Society},
     address={Providence, RI},
        date={2000},
      volume={75},
        ISBN={0-8218-2002-8},
}

\bib{Hjorth1996}{article}{
      author={Hjorth, Greg},
      author={Kechris, Alexander~S.},
       title={Borel equivalence relations and classifications of countable
  models},
        date={1996},
        ISSN={0168-0072},
     journal={Ann. Pure Appl. Logic},
      volume={82},
      number={3},
       pages={221\ndash 272},
}

\bib{Hjorth1998}{article}{
      author={Hjorth, Greg},
      author={Kechris, Alexander~S.},
      author={Louveau, Alain},
       title={Borel equivalence relations induced by actions of the symmetric
  group},
        date={1998},
        ISSN={0168-0072},
     journal={Ann. Pure Appl. Logic},
      volume={92},
      number={1},
       pages={63\ndash 112},
}

\bib{Kechris1995}{book}{
      author={Kechris, Alexander~S.},
       title={Classical descriptive set theory},
      series={Graduate Texts in Mathematics},
   publisher={Springer-Verlag},
     address={New York},
        date={1995},
      volume={156},
        ISBN={0-387-94374-9},
}

\bib{Lopez-Escobar1965}{article}{
      author={Lopez-Escobar, E. G.~K.},
       title={An interpolation theorem for denumerably long formulas},
        date={1965},
        ISSN={0016-2736},
     journal={Fund. Math.},
      volume={57},
       pages={253\ndash 272},
}

\bib{Melleray2010}{article}{
      author={Melleray, Julien},
       title={A note on {H}jorth's oscillation theorem},
        date={2010},
        ISSN={0022-4812},
     journal={J. Symbolic Logic},
      volume={75},
      number={4},
       pages={1359\ndash 1365},
         url={http://dx.doi.org/10.2178/jsl/1286198151},
      review={\MR{2767973}},
}

\bib{Nies.Solecki:15}{inproceedings}{
      author={Nies, A.},
      author={Solecki, S.},
       title={Local compactness for computable polish metric spaces is
  {$\mathbf{\Pi^1_1}$}-complete},
organization={Springer},
        date={2015},
   booktitle={Conference on {C}omputability in {E}urope},
       pages={286\ndash 290},
}

\bib{Scott1965}{incollection}{
      author={Scott, Dana},
       title={Logic with denumerably long formulas and finite strings of
  quantifiers},
        date={1965},
   booktitle={Theory of {M}odels ({P}roc. 1963 {I}nternat. {S}ympos.
  {B}erkeley)},
   publisher={North-Holland, Amsterdam},
       pages={329\ndash 341},
      review={\MR{0200133}},
}

\bib{Uspenskij2008}{article}{
      author={Uspenskij, Vladimir},
       title={On subgroups of minimal topological groups},
        date={2008},
        ISSN={0166-8641},
     journal={Topology Appl.},
      volume={155},
      number={14},
       pages={1580\ndash 1606},
         url={http://dx.doi.org/10.1016/j.topol.2008.03.001},
}

\bib{Vaught1974}{article}{
      author={Vaught, Robert},
       title={Invariant sets in topology and logic},
        date={1974/75},
        ISSN={0016-2736},
     journal={Fund. Math.},
      volume={82},
       pages={269\ndash 294},
        note={Collection of articles dedicated to Andrzej Mostowski on his
  sixtieth birthday, VII},
}

\bib{Zielinski2016}{article}{
      author={Zielinski, Joseph},
       title={The complexity of the homeomorphism relation between compact
  metric spaces},
        date={2016},
        ISSN={0001-8708},
     journal={Adv. Math.},
      volume={291},
       pages={635\ndash 645},
         url={http://dx.doi.org/10.1016/j.aim.2015.11.051},
      review={\MR{3459026}},
}

\end{biblist}
\end{bibdiv}

\end{document}